\documentclass[12pt,reqno]{amsart}
\usepackage{amsmath}
\usepackage{amsthm}
\usepackage{amssymb}
\usepackage{graphics}
\usepackage{latexsym}

\numberwithin{equation}{section}
\newtheorem{thm}{Theorem}[section]
\newtheorem{prop}[thm]{Proposition}
\newtheorem{lem}[thm]{Lemma}

\theoremstyle{remark}

\newcommand{\p}{\partial}

\title[W. P for the periodic Kawahara equation]{
}

\author[T. K. Kato]{
}
\email[Takamori Kato]{tk-kato@math.kyoto-u.ac.jp}


\begin{document}

\begin{center}
{\bf LOW REGULARITY WELL-POSEDNESS FOR THE PERIODIC KAWAHARA EQUATION}

\bigskip {\sc Takamori Kato}

\smallskip {\small Department of Mathematics, Kyoto University 

Kyoto, 606-8502, Japan}

\end{center}
\begin{abstract}
In this paper, we consider the well-posedness for the Cauchy problem of the Kawahara equation with low regularity data in the periodic case. 
We obtain the local well-posedness for $s \geq -3/2$ by a variant of Fourier restriction norm method introduced by Bourgain. 
Moreover, these local solutions can be extended globally in time for $s \geq -1$ by the I-method. 
On the other hand, we prove ill-posedness for $s<-3/2$ in some sense. 
This is a shape contrast to the results in the case of $\mathbb{R}$, where 
the critical exponent is equal to $-2$.  
\end{abstract}
\maketitle

\section{Introduction}

We consider the well-posedness for the Cauchy problem of the Kawahara equation which is one of the fifth order KdV type equations. 
\begin{align} \label{KT_1}
\begin{cases}
& \p_t u+ \alpha \p_{x}^5 u+\beta \p_x^3 u + \gamma \p_x ( u^2 )=0, \hspace{0.3cm} (t,x) \in [0,T] \times \mathbb{T}, \\
& u(0, x)=u_0(x), \hspace{0.3cm} x \in \mathbb{T},
\end{cases} 
\end{align}
where $\alpha, \beta, \gamma \in \mathbb{R}$ with $\alpha, \gamma \neq 0$ and 
$\mathbb{T}:= \mathbb{R}/ 2 \pi \mathbb{Z}$. 
Here the unknown function $u$ is assumed to 
be real valued or complex valued in the case we deal with 
the local well-posedness (LWP for short) and 
to be real valued when we consider the global well-posedness (GWP for short). 
By the renormalization of $u$, we may assume $\alpha=-1$, $\gamma=1$ and $\beta=-1,0$ or 
$1$. We put $v=u-a$ where $a$ is the integral mean value of initial data defined as $a:= \int_{\mathbb{T}} u_0 (x) dx$. 
If $u$ solves (\ref{KT_1}), then $v$ satisfies the following equation. 
\begin{align*} 
\begin{cases}
& \p_t v- \p_{x}^5 v+\beta \p_x^3 v + 2a \p_x v+ \p_x ( v^2 )=0, \hspace{0.3cm} (t,x) \in [0,T] \times \mathbb{T}, \\
& v(0, x)=v_0(x), \hspace{0.3cm} x \in \mathbb{T}.
\end{cases} 
\end{align*}
Note that the Fourier coefficient $\mathcal{F}_x (v) (0)$ of zero mode vanishes. 
It suffices to consider the well-posedness for (\ref{KT_1}) under the mean-zero  assumption 
$\int_{\mathbb{T}} u_0 (x) dx=0$ because the linear first order term is harmless. 
This observation was used by Bourgain \cite{Bo}. Without the mean-zero assumption, 
the data-to-solution map fails to be $C^2$ in $H^s(\mathbb{T})$ for any $s \in \mathbb{R}$. 
So this assumption is crucial for some of analysis that follows.  
From the above argument, we only consider the case $\dot{\mathbb{Z}} := \mathbb{Z} \setminus \{0\}$. 
The Kawahara equation models the capillary waves on a shallow layer and 
the magneto-sound propagation in plasma (see e.g. \cite{Ka}). 
This equation has solitary waves with $\beta=1$ and many conserved quantities. 
Our aim is to prove the well-posedness for (\ref{KT_1}) with low regularity data given in the Sobolev space $\dot{H}^s(\mathbb{T})$. 
Here $\dot{H}^s(\mathbb{T})$ is defined by the norm, 
\begin{align*}
\| u \|_{\dot{H}^s(\mathbb{T})} := 
\| \langle k \rangle^s \mathcal{F}_x u \|_{l_k^2( \dot{\mathbb{Z}} )},
\end{align*}
where $\langle \cdot \rangle:=(1+|\cdot |^2)^{1/2}$ . 
We first use the Fourier restriction norm method to prove LWP for (\ref{KT_1}). 
This method was introduced by Bourgain \cite{Bo}. Next, 
we extend local solutions to global-in-time ones by the I-method which 
was exploited by Colliander, Keel, Staffilani, Takaoka and Tao \cite{CoKe}, \cite{I02}. 
The Kawahara equation with the periodic boundary does not have the Kato smoothing effect 
unlike the case of $\mathbb{R}$, though a weak version of the Strichartz estimate still holds in the periodic setting. 
It is possible to make a close investigation into the resonance of nonlinear interactions under the periodic 
boundary conditions.

The local well-posedness for the periodic KdV equation has been extensively studied. 
Bourgain \cite{Bo} proved LWP in $\dot{H}^s$ for $s \geq 0$. 
Kenig, Ponce and Vega \cite{KPV96} refined Bourgain's argument to show LWP in $\dot{H}^s$ for $s>-1/2$. 
Moreover, Colliander, Keel, Staffilani, Takaoka and Tao \cite{CoKeSt} obtained LWP in the critical case $s=-1/2$. 
On the other hand, Christ, Colliander and Tao \cite{CCT} showed that  the data-to-solution map fails to be uniformly 
continuous for $-2< s <-1/2$. 
We now recall the local well-posedness results for the Kawahara equation. 
Hirayama \cite{Hi} proved LWP in $\dot{H}^s$ for $s \geq -1$ in the periodic case, 
which was an adaptation of the argument to Kenig, Ponce and Vega \cite{KPV96}. 
Moreover, there are many studies in the case of $\mathbb{R}$. 
Chen and Guo \cite{CG} proved for $s \geq -7/4$, using some modified Bourgain space $\bar{F}^s$ introduced in \cite{Gu}. 
Following an idea of Bejenaru  and Tao \cite{BT} and Kishimoto and Tsugawa \cite{KT}, 
we improved the previous results to $s \geq -2$ in \cite{TK_K}. This result is optimal in such a sense that 
the data-to-solution map fails to be continuous when $s<-2$. 
Earlier results can be found in \cite{CLMW}, \cite{CDT} and \cite{WCD}. 
The main difficulty in obtaining LWP for the periodic equation is to recover no derivatives by the smoothing effects.  
So we need to make a more complex modification of the Bourgain space. 
Then we find a suitable modification of function spaces to obtain the following theorem. 
\begin{thm} \label{LWP_TK}
Let $s \geq -3/2$. Then (\ref{KT_1}) is locally well-posed in $\dot{H}^{s}(\mathbb{T})$. 
\end{thm}

On the other hand, we obtain the ill-posedness result in the following sense.
\begin{thm} \label{ill_TK}
Let $s<-3/2$. Then, there is no $T>0$ such that the flow map, 
$\dot{H}^s(\mathbb{T}) \ni u_0 \mapsto u(t) \in \dot{H}^s(\mathbb{T})$, can be $C^3$ for 
any $t \in (0,T]$  
\end{thm}

These theorems imply that the critical regularity is $s=-3/2$. 
Moreover, the local solutions 
obtained in Theorem~\ref{LWP_TK} are shown to exist on an arbitrary time by the I-method. 
Colliander, Keel, Staffilani, Takaoka and Tao \cite{CoKeSt} proved GWP for the periodic case KdV equation when $s>-1/2$,  
which was improved to $s \geq -1/2$ in \cite{CKSTT03}. 
We now describe the global well-posedness results for the Kawahara equation 
in the non-periodic case.   
Note that it is difficult to apply the I-method to the Kawahara equation because this equation has less symmetries than 
the KdV equation. Chen and Guo \cite{CG} overcame this issue and used the similar argument to \cite{CoKeSt} 
to show GWP for $s \geq -7./4$. Recently, we have refined their argument and established GWP for 
$s \geq -38/21$ in \cite{TK_G}.  
We apply the argument presented for the non-periodic case to the periodic setting so that the following is established. 
\begin{thm} \label{GWP_TK}
Let $s \geq -1$. Then (\ref{KT_1}) is globally well-posed in $\dot{H}^s (\mathbb{T})$.
\end{thm}

This result is optimal as long as we use the standard Bourgain space. 

\vspace{0.3em}

We now use the scaling argument. For $\lambda \geq 1$, 
\begin{align*}
u_{\lambda}(t,x):= \lambda^{-4} u(\lambda^{-5}t, \lambda^{-1} x), 
\hspace{0.3cm} u_{0,\lambda}(x):=\lambda^{-4} u_0(\lambda^{-1} x)
\end{align*}
If $u$ solves (\ref{KT_1}), $u_{\lambda}$ satisfies 
the following rescaled Cauchy problem;
\begin{align} \label{KT_2}
\begin{cases}
& \p_t u_{\lambda} -\p_x^5 u_{\lambda} 
+ \ \lambda^{-2} \beta \p_x^3 u_{\lambda} +  \p_x(u_{\lambda}^2 )=0, 
\hspace{0.3cm} (t,x)  \in [0, \lambda^5 T] \times \mathbb{T}_{\lambda}, \\
& u_{\lambda}(0,x)=u_{0,\lambda} (x), 
\hspace{0.3cm} x \in \mathbb{T}_{\lambda}, 
\end{cases}
\end{align}
where $\mathbb{T}_{\lambda}:= \mathbb{R} / 2 \pi \lambda \mathbb{Z}$. $\widehat{\varphi}$ denotes 
the Fourier transform on $\mathbb{T}_{\lambda}$ of $\varphi$ as follows; 
\begin{align*}
\widehat{\varphi}(k) := \frac{1}{ \sqrt{2 \pi}} \int_0^{2 \pi \lambda} 
e^{-ikx} \varphi(x) dx, \hspace{0.3cm} k \in \dot{\mathbb{Z}}_{\lambda}
:= \frac{1}{\lambda} \dot{\mathbb{Z}}.
\end{align*}
Here the space $\dot{H}^{s} (\mathbb{T}_{\lambda})$ is equipped with the norm 
\begin{align*}
\| \varphi \|_{\dot{H}^s(\mathbb{T}_{\lambda})}:=
\| \langle k \rangle^{s} \widehat{\varphi} \|_{l_k^2 (\dot{\mathbb{Z}}_{\lambda})},
\end{align*}
where 
\begin{align*}
\| f \|_{l_k^p(\dot{\mathbb{Z}}_{\lambda})}:= \Bigl( \frac{1}{\lambda}
\sum_{k \in \dot{\mathbb{Z}}_{\lambda}} |f (k)|^{p} \Bigr)^{1/p},
\end{align*}
for $1 \leq p \leq \infty$. 
A direct calculation shows that
\begin{align} \label{in_sm}
\| u_{0,\lambda} \|_{\dot{H}^s (\mathbb{T}_{\lambda})} 
\leq \lambda^{-7/2-s} \| u_0 \|_{ \dot{H}^s (\mathbb{T})}
~~ \text{for} ~~ s<0.   
\end{align}
Therefore we can assume smallness of initial data.  So it suffices to solve (\ref{KT_2}) for sufficiently small data. 
We first summarize the local well-posedness theory.   
The main idea is how to define the function space to construct solutions.  
When $s$ is small, especially negative, the Bourgain space plays an important role. 
The Bourgain space $X^{s,b}(\mathbb{R} \times \mathbb{T}_{\lambda})$ for $2 \pi \lambda$-periodic 
is defined by the norm
\begin{align*}
\| u \|_{X^{s,b} (\mathbb{R} \times \mathbb{T_{\lambda}})}
:= \Bigl\| \langle k \rangle^s \langle \tau-p_{\lambda}(k) \rangle^b \widehat{u}
\Bigr\|_{l_k^2(\dot{\mathbb{Z}}_{\lambda}   ; L_{\tau}^2 (\mathbb{R}) )}, 
\end{align*}
where $p_{\lambda}(k):= k^{5}+\beta \lambda^{-2} k^3$. 
Remark that 
the Bourgain space depends on the linear part of our target equation. 
One of the key estimates is the bilinear estimate in $X^{s,b}$ as follows:
\begin{align} \label{BE_T1}
\| \Lambda^{-1} \p_x (uv) \|_{X^{s,b}} \leq C 
\| u \|_{X^{s,b}} \| v \|_{X^{s,b}},
\end{align}
where $\Lambda^{b}$ is the Fourier multiplier defined as 
$\Lambda^{b}:=
\mathcal{F}_{\tau,k}^{-1} \langle \tau-p_{\lambda}(k) \rangle^{b} \mathcal{F}_{t,x} 
$ for $b \in \mathbb{R}$. 
From the bilinear estimate and some linear estimates, the standard argument of the Fourier restriction norm method works to obtain LWP. 
Hirayama \cite{Hi} showed (\ref{BE_T1}) for $s \geq -1$. On the other hand, he proved that this estimate fails for any $b \in \mathbb{R}$ when $s<-1$. 
So it is difficult to construct the local solutions by the iteration argument when $s<-1$. 
To avoid this difficulty, we modify the Bourgain space $X^{s,b}$ to control strong nonlinear interactions and 
establish the bilinear estimate at the critical regularity $s=-3/2$. 
An idea of a modification of $X^{s,b}$ was developed by Bejenaru and Tao \cite{BT}. 
They considered the quadratic Schr\"{o}dinger equation with the nonlinearity $u^2$ and 
obtained LWP in the critical case $H^{-1}(\mathbb{R})$. 
Note that there is no general framework for modifying $X^{s,b}$. 
This is one of the most difficult points in our study. 
Compared to the non-periodic case, less derivatives can be recovered by the smoothing effects in the 
periodic setting. 
So nonlinear interactions which we can ignore in the non-periodic case take effect. 
Therefore we need to make a more complex modification of $X^{s,b}$ to control three types nonlinear interactions. 
We now mention how to modify $X^{s,b}$. 
From the counterexamples of (\ref{BE_T1}) in the case $s<-1$, we find the regions in which strong nonlinear interactions appear. 
In these domains, 
we make a suitable modification of $X^{s,b}$ as follows; 
\begin{align*}
\| u \|_{Z^s} & :=  \| P_{D_1} u \|_{X^{s,3/4}}+ \| P_{D_2} u \|_{X^{-3s-1,s+1}} \\
&~ + \| P_{D_3} u  \|_{X^{-s/2-1,s/2+1}} + \|  u \|_{Y^s},  \hspace{0.3cm} \text{for} 
\hspace{0.3cm}  -3/2 \leq s \leq -1,
\end{align*}
where $P_{\Omega}$ is the Fourier projection onto a set 
$\Omega \subset \mathbb{R} \times \dot{\mathbb{Z}}_{\lambda}$ 
and 
\begin{align*} 
D_1 &:= \bigl\{ (\tau, k) \in \mathbb{R} \times \dot{\mathbb{Z}}_{\lambda}~;~
|\tau -p_{\lambda} (k) | \leq  |k|^4/10  ~~\text{and}~~ |k| \geq 1 \bigr\}, \\   
D_2 &:= \bigl\{ (\tau, k) \in \mathbb{R} \times \dot{\mathbb{Z}}_{\lambda}~;~  |k|^4/10
 \leq |\tau -p_{\lambda} (k) | \leq  |k|^5/10 ~~\text{and}~~ |k| \geq 1  \bigr\}, \\  
D_3 &:= \bigl\{ (\tau, k) \in \mathbb{R} \times \dot{\mathbb{Z}}_{\lambda}~;~
|\tau -p_{\lambda} (k) | \geq  |k|^5/10 ~~\text{and}~~ \frac{1}{\lambda} \leq |k| \leq 1  \bigr\}.
\end{align*}
Here $\| u \|_{Y^s} := \| \langle k \rangle^{s} \widehat{u} \|_{l_k^2 L_{\tau}^1}$ 
and $Y^s$ is continuously embedded into $C(\mathbb{R} ; \dot{H}^s (\mathbb{T}_{\lambda}))$. 
Using the function space above, 
we obtain the following bilinear estimate which is one of the main estimates in the present paper. 
 \begin{prop} \label{prop_BE_2}
Let $-3/2 \leq s <-1$. Then, the following estimate holds. 
\begin{align} \label{BE_T2} 
\| \Lambda^{-1} \p_x (uv) \|_{Z^s} \leq C \| u \|_{Z^s} \| v \|_{Z^s},  
\end{align}
where a positive constant $C$ is independent of $\lambda$. 
\end{prop}

\vspace{0.5em}

Next, we extend the local solution obtained above globally in time. 
In the case $s$ is negative, we have no conservation laws. To avoid this difficulty, 
we apply the I-method exploited by Colliander, Keel, Staffilani, Takaoka and Tao \cite{CoKe}, \cite{I02}. 
The main idea is to use a modified energy defined for less regular functions, which is not conserved. 
If we control the growth of the modified energy in time, this enables us to 
iterate the local theory to continue the solution to any time $T$. 
We now mention the definition of the modified energy $E_I^{(2)} (u)$. 
The operator $I: H^s \rightarrow L^2$ is the Fourier multiplier defined as 
$I= \mathcal{F}_{\xi}^{-1} m (\xi) \mathcal{F}_x$. 
Here $m$ is a smooth and monotone function satisfying 
\begin{align*}
m(\xi):=
\begin{cases}
1 ~~& \text{for} ~~ |\xi| \leq N \\
|\xi|^{s} N^{-s} ~~ & \text{for} ~~|\xi| \geq 2N,
\end{cases}
\end{align*} 
for $s<0$ and $N \gg1$. The modified energy $E_I^{(2)} (u)$ is defined as 
$E_{I}^{(2)} (u)(t) := \| I u (t) \|_{L^2}^2 $. In the I-method, the key estimate is 
the almost conservation law which implies the increment of the modified energy 
is sufficiently small for a short time interval and large $N$. 
Following the argument of \cite{CoKe}, we obtain the almost conservation law and show GWP for 
$s > -21/26$. However, the growth of the modified energy $E_I^{(2)} (u)$ in time cannot be controlled for 
$-1 \leq s \leq -21/26$. 
Then we add some correction terms to the original modified energy 
$E_I^{(2)} (u)$ to construct a new modified energy in order to remove some oscillations in this functional. 
This idea was developed by Colliander, Keel, Staffilani, Takaoka and Tao \cite{CoKeSt}. They \cite{CoKeSt} proved 
GWP of the KdV equation for $s>-3/4$ in the case of $\mathbb{R}$ and for $s>-1/2$ in the periodic 
case. Chen and Guo \cite{CG} establish the sharp upper bound of some multiplier to show GWP of the Kawahara equation 
for $s \geq -7/4$ in the case of $\mathbb{R}$. 
Following the argument of \cite{CG}, we obtain the almost conservation law for 
the modified energy $E_I^{(4)} (u)$ by adding two suitable correction terms to 
the original functional when $s \geq -1$. 
On the other hand, the difference between the almost conserved quantities $E_I^{(4)} (u)$ and 
the first modified energy $E_I^{(2)} (u)$ can be controlled by $E_I^{(2)}(u)$ when the time is fixed. 
This estimate and the almost conservation law imply that 
the well-posedness on any time interval. Remark that 
we do not expect to recover any derivatives by the bilinear Strichartz estimate in the periodic case 
(see Lemma~\ref{lem_L_4} in section 2). This is the reason why it is hard so that 
the I-method is applicable when $s<-1$. 

\vspace{0.5em}

We use the following notations in this paper. 
$A \lesssim B$ means $A \leq C B$ for some positive constant $C$ 
and $A \sim B$ when both $A \lesssim B$ and $B \lesssim A$. 
$c+$ means $c+\varepsilon$, while $c-$ means $c- \varepsilon$ where 
$\varepsilon>0$ is enough small. 
For a normed space $\mathcal{X}$ and a set $\Omega$, 
$\| \cdot \|_{\mathcal{X}(\Omega) }$ denotes $\| f \|_{\mathcal{X} (\Omega)} := \| \chi_{\Omega} f \|_{\mathcal{X}}$ 
where $\chi_{\Omega}$ is the characteristic function of $\Omega$. 

The rest of this paper is planned as follows. In Section 2, we give some preliminary lemmas. 
In Section 3, we prove the bilinear estimate (\ref{BE_T2}) and
 give the proof of LWP in Section 4.  
In Section 5, we show GWP by the I-method, following \cite{CG} and \cite{CoKeSt}
In Section 6, we give the proof of Theorem~\ref{ill_TK} which is based on Bourgain's work \cite{Bo97}. 

\vspace{1em}
 
\noindent
\textbf{Acknowledgment.} The author would like to appreciate his adviser Professor Yoshio Tsutsumi for many helpful conversation and encouragement and 
thank Professor Kotaro Tsugawa and Professor Nobu Kishimoto for helpful comments.

\section{Preliminaries}
In this section, we prepare the bilinear Strichartz estimate to show the main estimates. When we use the variables $(\tau,k)$, $(\tau_1,k_1)$ and $(\tau_2, k_2)$, we always assume the relation
\begin{align*}
(\tau, k)=(\tau_1, k_1) + (\tau_2, k_2).
\end{align*}
The bilinear estimate (\ref{BE_T2}) can be established by the H\"{o}lder inequality, the Young inequality and the following estimate.     

\begin{lem} \label{lem_L_4} 
If $b, b' \in \mathbb{R}$ satisfy $b+b' \geq 29/40$ and $b, b' > 9/40$, then we have 
\begin{align} \label{es_L_4_1}
\bigl\| P_{\{ |k| \geq 1  \}} ( uv ) \bigr\|_{L_{t,x}^2} & \lesssim \| u \|_{X^{0,b}} \| v \|_{X^{0,b'}}, \\ 
\label{es_L_4_2}   
\bigl\| u ( P_{\{ |k| \geq 1 \}} v) \bigr\|_{X^{0,-b'}} & \lesssim \| u \|_{X^{0,b}} \| v \|_{L_{t,x}^2}.  
\end{align}
\end{lem}
\begin{proof}
For a dyadic number $M \geq 1$, $u_M$ denotes that the support of $\widehat{u}$ is restricted to 
the dyadic block $\{ \langle \tau-p_{\lambda}(k)  \rangle \sim M \}$. We use the triangle inequality 
and the Plancherel theorem to have 
\begin{align*}
\bigl\| P_{\{ |k| \geq 1 \}} (uv)  \bigr\|_{L_{t,x}^2} & \lesssim 
\sum_{M_1,M_2 \geq 1} \bigl\| P_{\{ |k| \geq 1 \}} (u_{M_1} v_{M_2} ) \bigr\|_{L_{t,x}^2} \\
& \sim \sum_{M_1, M_2 \geq 1} \Bigl\| \frac{1}{ \lambda} \sum_{k_1 \in \mathbb{Z}_{\lambda}}
\int_{\mathbb{R}}  \widehat{u}_{M_1}(\tau_1, k_1) \widehat{v}_{M_2} (\tau_2 ,k_2) 
d\tau_1 \Bigr\|_{l_k^2 L_{\tau}^2 (|k| \geq 1)} .
\end{align*}  
Using the Schwarz inequality twice, the above is bounded by 
\begin{align} \label{es_dyb}
\sum_{M_1, M_2 \geq 1}  \sup_{(\tau,k) \in \mathbb{R} \times \mathbb{Z}_{\lambda} }
\Bigl( \frac{1}{\lambda} 
\sum_{k_1 \in \mathbb{Z}_{\lambda}} \int_{\mathbb{R}} \chi_E(\tau,k,\tau_1,k_1)
d\tau_1  \Bigr)^{1/2} 
\| u_{M_1} \|_{L_{t,x}^2} \| v_{M_2} \|_{L_{t,x}^2},
\end{align}
where 
\begin{align*}
E:=\bigl\{ (\tau,k,\tau_1, k_1)  \in ( \mathbb{R} \times \mathbb{Z}_{\lambda} )^2 
~;~ |\tau_1-p_{\lambda} (k_1) | \sim M_1, ~ |\tau_2-p_{\lambda} (k_2)| \sim M_2
, ~|k| \geq 1  \bigr\}. 
\end{align*}
We now show the following estimate when 
$M_1 \geq M_2$. 
\begin{align} \label{es_dy_T1}
\sup_{(\tau,k) \in \mathbb{R} \times \mathbb{Z}_{\lambda}} \frac{1}{\lambda} 
\sum_{k_1 \in \mathbb{Z}_{\lambda}} \int_{\mathbb{R}} \chi_E(\tau,k,\tau_1,k_1)
d\tau_1  \lesssim M_1^{9/20} M_2. 
\end{align}
The identity, 
\begin{align*}
& \bigl( \tau- \frac{k^5}{16}- \beta \lambda^{-2} \frac{k^3}{4} \bigr)-
(\tau_1-p_{\lambda}(k_1)) -(\tau_2-p_{\lambda}(k_2)) \\
& \hspace{1.2cm} = \frac{5}{16} k (k_1-k_2)^2 
\bigl\{ (k_1-k_2)^2+ 2 k^2 + \frac{12}{5} \beta \lambda^{-2} \bigr\}, 
\end{align*}
implies 
\begin{align*}
& (k_1-k_2)^2= \\
&~~ \Bigl\{ \frac{ L_0+O(\max \{M_1, M_2 \}) }{|k|} +(k^2 +\frac{6}{5} \beta \lambda^{-2})^2 \Bigr\}^{1/2}-(k^2 + \frac{6}{5} \beta \lambda^{-2}),
\end{align*}
where $ \displaystyle  L_0:= \frac{16}{5} \bigl| \tau- \frac{k^5}{16}-\beta \lambda^{-2} \frac{k^3}{4}  \bigr|$. 
Now $(\tau,k)$ is fixed. Then the variation of $k_1$ is bounded by 
\begin{align} \label{es_leng1}
& \lambda  \frac{\max \{ M_1, M_2\}}{|k|}  
\Bigl\{ \frac{L_0+ O(\max\{ M_1, M_2 \}  )}{|k|}+
(k^2+ \frac{6}{5} \beta \lambda^{-2} )^2  \Bigr\}^{-1/2} \nonumber \\
\times & \Bigl[ \Bigl\{\frac{L_0+ O(\max \{M_1,M_2  \}) }{|k|} +(k^2+ \frac{6}{5} \beta \lambda^{-2} )^{2} \Bigr\}^{1/2} -(k^2+ \frac{6}{5} \beta \lambda^{-2} ) \Bigr]^{-1/2}. 
\end{align}
Note that for $|k| \geq 1$
\begin{align} \label{es_leng2}
& \Bigl[ \Bigl\{\frac{L_0+ O(\max \{M_1,M_2  \}) }{|k|} +(k^2+ \frac{6}{5} \beta \lambda^{-2} )^{2} \Bigr\}^{1/2} -(k^2+ \frac{6}{5} \beta \lambda^{-2} ) \Bigr]^{-1/2} 
\nonumber \\
& \hspace{1cm} \gtrsim |k|^{-3/2} O(\max \{M_1, M_2 \})^{1/4}.
\end{align}
We apply (\ref{es_leng2}) and the Young inequality to (\ref{es_leng1}) so that the variation of $k_1$ is at most 
\begin{align*}
\lambda |k|^{3/2-5/2p} \max\{M_1, M_2 \}^{3/4-1/2p} \hspace{0.3cm} \text{for}  \hspace{0.3em} 
1<p<\infty.
\end{align*} 
The above is equal to $\lambda \max\{M_1,M_2 \}^{9/20}$ with $p=5/3$. 
If we also fix $k_1$, $\tau_1$ is restricted to the interval of measure 
$O(\max\{ M_1,M_2 \})$. Therefore we obtain (\ref{es_dy_T1}) when 
$M_1 \geq M_2$. 
Substituting (\ref{es_dy_T1}) into (\ref{es_dyb}), we have
\begin{align*}
\| P_{\{ |k| \geq 1 \}} uv \|_{L_{t,x}^2} \lesssim &
\sum_{M_1, M_2 \geq 1} M_1^{9/40} M_2^{1/2} \| u_{M_1} \|_{L_{t,x}^2}
\| v_{M_2} \|_{L_{t,x}^2} \\
= & \sum_{N, M_2 \geq 1} M_2^{29/40} N^{9/40} \| u_{N M_2} \|_{L_{t,x}^2} 
\| v_{M_2} \|_{L_{t,x}^2} \\
\lesssim & 
\sum_{N \geq 1} \sum_{M_2 \geq 1} N^{9/40-b} M_2^{29/40-(b+b')}
(N M_2)^{b} \| u_{N M_2} \|_{L_{t,x}^2} M_2^{b'} \| v_{M_2} \|_{L_{t,x}^2}.
\end{align*}
Applying the Schwarz inequality in $M_2$ and summing over $N$, 
we obtain the desired estimate. 

On the other hand, we immediately obtain (\ref{es_L_4_2}) from the duality argument. 
\end{proof}

We put  a one parameter semigroup $U_{\lambda}(t)$ as follow:
\begin{align*}
U_{\lambda}(t):= \mathcal{F}_{ k }^{-1} \exp(i p_{\lambda}(k) t ) \mathcal{F}_{x}. 
\end{align*}
For any time interval $I$, we define the restricted space $Z^s (I)$ by the norm 
\begin{align*}
\| u \|_{Z^s(I)}:= \inf \bigl\{ \| v \|_{Z^s} ~;~ 
u(t)=v(t) ~~\text{on} ~~t \in I  \bigr\}. 
\end{align*}
From the definition, $Z^s ([0,T])$ has the property as follows; 
\begin{align*}
X^{s,3/4} ([0,T]) \hookrightarrow 
Z^s ([0,T]) \hookrightarrow C([0,T]; \dot{H}^s (\mathbb{T}_{\lambda})) .
\end{align*}
The above property implies the following linear estimates. 
\begin{prop} \label{prop_linear1}
Let $s \in \mathbb{R}$, $T>0$ and $\lambda \geq 1$. Then, we have 
\begin{align*} 
\| U_{\lambda}(t) u_0 \|_{Z^s([0,T]) } \lesssim  \| u_0 \|_{\dot{H}^s(\mathbb{T}) } . 
\end{align*}
\end{prop}

\begin{prop} \label{prop_linear2}
Let $s \in \mathbb{R}$, $T>0$ and $\lambda \geq 1$. If the bilinear estimate 
(\ref{BE_T2}) holds, then we have
\begin{align*}
\bigl\| \int_0^t U_{\lambda}(t-t') F(t') dt'   \bigr\|_{Z^s ([0,T]) } 
\lesssim  \| u \|_{Z^s ([0,T])} 
\| v \|_{Z^s([0,T])}. 
\end{align*}
\end{prop}

For the proofs of these propositions, see \cite{BT}.

\section{Proof of the bilinear estimate}

\noindent
In this section, we give a proof of the bilinear estimate (\ref{BE_T2}). For simplicity, we introduce the Fourier multiplier 
$J^{\sigma}:= \mathcal{F}_{k}^{-1} \langle k \rangle^{\sigma} \mathcal{F}_x$ for $\sigma \in \mathbb{R}$. 
Proposition~\ref{prop_BE_2} can be established by H\"{o}lder's and Young's inequalities and Lemma~\ref{lem_L_4}.  

\begin{proof}[Proof of Proposition~\ref{prop_BE_2}]
We prove the following two estimates to obtain (\ref{BE_T2}).
\begin{align}
\label{BE_X}
\| \Lambda^{-1} \p_x (uv) \|_{X_{w}^s} & \lesssim 
\| u \|_{Z^s} \| v \|_{Z^s}, \\
\label{BE_Y}
\| \Lambda^{-1} \p_{x} (uv) \|_{Y^s} & \lesssim 
\| u \|_{Z^s} \| v \|_{Z^s},
\end{align}
where $\| \cdot \|_{X_w^s}$ is the norm 
removing $\| \cdot \|_{Y^s} $ from $\| \cdot \|_{Z^s}$. 
Firstly, we divide $(\mathbb{R} \times \dot{\mathbb{Z}}_{\lambda})^2$ into 
six parts as follows;
\begin{align*}
\Omega_0 & :=\bigl\{ (\tau,k,\tau_1,k_1) \in (\mathbb{R} \times \dot{\mathbb{Z}}_{\lambda})^2~;~ 
|k|, |k_1| \lesssim 1 \bigr\}, \\
\Omega_1 & := \bigl\{ (\tau,k,\tau_1,k_1) \in (\mathbb{R} \times \dot{\mathbb{Z}}_{\lambda})^2 \setminus \Omega_0~;~ 
 |k_1| \sim |k-k_1| \gg |k| \geq 1 \bigr\}, \\
\Omega_2 & := \bigl\{ (\tau,k,\tau_1,k_1) \in (\mathbb{R} \times \dot{\mathbb{Z}}_{\lambda})^2 \setminus \Omega_0~;~ 
 |k_1| \sim |k-k_1| \gg |k|~\text{and}~ 1 \geq |k| \geq 1/\lambda \bigr\}, \\
\Omega_3 & := \bigl\{ (\tau,k,\tau_1,k_1) \in (\mathbb{R} \times \dot{\mathbb{Z}}_{\lambda})^2 \setminus \Omega_0~;~ 
 |k| \sim |k-k_1| \gg |k_1| \geq 1 \bigr\}, \\
\Omega_4& := \bigl\{ (\tau,k,\tau_1,k_1) \in (\mathbb{R} \times \dot{\mathbb{Z}}_{\lambda})^2 \setminus \Omega_0~;~ 
 |k| \sim |k-k_1| \gg |k_1 | ~\text{and}~1 \geq |k_1| \geq 1/\lambda \bigr\}, \\
\Omega_5 & := \bigl\{ (\tau,k,\tau_1,k_1) \in (\mathbb{R} \times \dot{\mathbb{Z}}_{\lambda})^2 \setminus \Omega_0~;~ 
 |k| \sim |k_1| \sim |k-k_1| \geq 1 \bigr\}.
 \end{align*}
Recall that $Z^s$ has the following properties;
\begin{align*} 
\| u \|_{X^{s,1/4}} \lesssim \| u \|_{Z^{s}} \lesssim \| u \|_{X^{s,3/4}}
~~ \text{and} ~~\| u \|_{X^{s,1/2}(D_1 \cup D_2)} \lesssim \| u \|_{Z^{s}(D_1 \cup D_2) }. 
\end{align*}
\underline{\bf Estimate in $\Omega_0$}

From the property of $Z^s$, we only estimate the norm $X^{s,3/4}$ of $\Lambda^{-1} \p_x (uv)$. 
From $|k|, |k_1|, |k-k_1| \lesssim 1$, we use the H\"{o}lder inequality and the Young inequality to have 
\begin{align*}
\| |k| \langle \tau-p_{\lambda} (k) \rangle^{-1/4} \widehat{u}* \widehat{v} \|_{l_k^2 L_{\tau}^2}
\lesssim & \| |k| \|_{l_k^2} \| \widehat{u}* \widehat{v} \|_{l_k^{\infty} L_{\tau}^2} \\
\lesssim & \| \widehat{u} \|_{l_k^2 L_{\tau}^2} \| \widehat{v} \|_{l_k^2 L_{\tau}^1},
\end{align*}
which is an appropriate bound. 

Here we put 
$L_{\max}= \max \{ |\tau-p_{\lambda}(k)|, |\tau_1-p_{\lambda}(k_1)|, | (\tau-\tau_1) -p_{\lambda}(k-k_1)|
\}$. In the remainder case, we often use the algebraic relation as follows; 
\begin{align} \label{alg}
L_{\max} \geq & 
\frac{1}{3} \Bigl| (\tau-p_{\lambda}(k))- (\tau_1-p_{\lambda}(k_1)) -\bigl\{ (\tau-\tau_1)-p_{\lambda}(k-k_1) \bigr\} \Bigr| \nonumber \\
\geq &
\frac{5}{6} \Bigl|k k_1 (k-k_1) \bigl\{ k^2+k_1^2+(k-k_1)^2+ \frac{6}{5} \beta
 \lambda^{-2} \bigr\} \Bigr| .
\end{align} 

\vspace{0.3em}

\noindent
(I) We prove the estimate for $\Omega_1$. We first decompose $\Omega_1$ into three parts as follows;
\begin{align*}
\Omega_{11}& := \bigl\{ (\tau,k,\tau_1,k_1) \in \Omega_1~; ~ |\tau-p_{\lambda}(k)|=L_{\max} \bigr\}, \\
\Omega_{12}& := \bigl\{ (\tau,k, \tau_1,k_1) \in \Omega_1~; ~ |\tau_1-p_{\lambda}(k_1)|=L_{\max} \bigr\}, \\
\Omega_{13}& := \bigl\{ (\tau,k, \tau_1,k_1) \in \Omega_1~; ~ |(\tau-\tau_1)-p_{\lambda}(k-k_1)|=L_{\max} \bigr\}. 
\end{align*}
The case $\Omega_{13}$ is identical to the case $\Omega_{12}$. So we omit this case. 
Note that $L_{\max} \gtrsim |k k_1^4|$ in $\Omega_1$ from (\ref{alg}). 

\vspace{0.3em}

\noindent
(Ia) In $\Omega_{11}$, 
$\widehat{u}* \widehat{v}$ is supported on $D_3$ from the definition. We use Lemma~\ref{lem_L_4} with $b'=1/4$ and 
$b=1/2$ to obtain
\begin{align*}
\| \langle k \rangle^{-s/2} \langle \tau-p_{\lambda} (k) \rangle^{s/2} \widehat{u}* \widehat{v} \|_{l_k^2 L_{\tau}^2}
 \lesssim \| J^s u J^s v \|_{L_{t,x}^2} \lesssim \| u \|_{X^{s,1/2}} \| v \|_{X^{s,1/4}},
\end{align*}
which implies the desired estimate except the case $\widehat{u}$ and $\widehat{v}$ are restricted to $D_3$. 
Next we consider the case both $\widehat{u}$ and $\widehat{v}$ are supported on $D_3$. We use H\"{o}lder's inequality 
and Young's inequality to have
\begin{align*}
& \| \langle k \rangle^{-s/2} \langle \tau-p_{\lambda}(k) \rangle^{s/2} \widehat{u} * \widehat{v} \|_{l_k^{2} L_{\tau}^2}
\lesssim \| \widehat{u} * (\langle k \rangle^{2s} \widehat{v} ) \|_{l_k^2 L_{\tau}^2} \\
& \hspace{0.3cm} \lesssim \| (\langle k \rangle^{-s/2-1} \langle \tau-p_{\lambda}(k) \rangle^{s/2+1} \widehat{u})* 
(\langle k \rangle^{-4} \widehat{v}) \|_{l_k^{2} L_{\tau}^2}
\lesssim \| u \|_{X^{-s/2-1,s/2+1}} \| \langle k \rangle^{-4} \widehat{v} \|_{l_k^1 L_{\tau}^1},
\end{align*}
which shows the required estimate since $\| \langle k \rangle^{-4} \widehat{v} \|_{l_k^1 L_{\tau}^1}
\lesssim \| v \|_{Y^s}$ by the Schwarz inequality. 
Moreover we estimate the norm $Y^s$ of $\Lambda^{-1} \p_x(uv)$. Following 
$|\tau-p_{\lambda}(k)| \gtrsim |k k_1^4|$, we use the H\"{o}lder inequality and the Young inequality to obtain
\begin{align*}
\| \langle k \rangle^{s+1} \langle \tau-p_{\lambda}(k) \rangle^{-1} \widehat{u}*\widehat{v} \|_{l_k^2 L_{\tau}^1} \lesssim 
\| (\langle k \rangle^{-2} \widehat{u}) *( \langle k \rangle^{-2} \widehat{v} ) \|_{l_k^{\infty} L_{\tau}^1} \lesssim \| u \|_{Y^s} \| v \|_{Y^s}. 
\end{align*}

\vspace{0.3em}

\noindent
(Ib) We show the estimate for $\Omega_{12}$. Consider three subregions
\begin{align*}
\Omega_{12a}& := \bigl\{ (\tau,k, \tau_1,k_1) \in \Omega_{12}~; ~ |\tau_1-p_{\lambda}(k_1)| \sim |k k_1^4|  \bigr\}, \\
\Omega_{12b}& := \bigl\{ (\tau,k, \tau_1,k_1) \in \Omega_{12}~; ~ |k_1|^5 \gtrsim |\tau_1-p_{\lambda}(k_1)| \gg |k k_1^4| \bigr\}, \\
\Omega_{12c}& := \bigl\{ (\tau,k, \tau_1,k_1) \in \Omega_{12}~; ~ |\tau_1-p_{\lambda}(k_1)| \gtrsim |k_1|^5 \bigr\}, 
\end{align*}

In $\Omega_{12a}$, $\widehat{u}$ is restricted to $D_2$. Then we use Lemma~\ref{lem_L_4} with $b'=1/4$ and $b=1/2$ to obtain 
\begin{align*}
\| \langle k \rangle^s \langle \tau-p_{\lambda}(k) \rangle^{-1/4} \widehat{u}*\widehat{v} \|_{l_k^2 L_{\tau}^2} 
\lesssim & \| (J^{-3s-1} \Lambda^{s+1} u) (J^{-s-3} v )\|_{X^{0,-1/4}} \\
\lesssim & \| u \|_{X^{-3s-1,s+1}} \| v \|_{X^{s,1/2}}.
\end{align*}

In $\Omega_{12b}$ and $\Omega_{12c}$, either $|\tau_1-p_{\lambda}(k_1)| \sim |\tau-p_{\lambda}(k) |$ or 
$|\tau_1-p_{\lambda}(k_1)| \sim |(\tau-\tau_1)- p_{\lambda}(k-k_1)| $ happens. 
The former case is almost identical to the case (Ia).  
So we only consider the latter case. 

We prove the estimate for $\Omega_{12b}$. 
In this case, we may assume that both $\widehat{u}$ and $\widehat{v}$ are supported on 
$D_2$ and $|\tau-p_{\lambda} (k)| \lesssim |k_1|^5$. 
From $\langle k_1 \rangle^{3s+1} \langle \tau_1-p_{\lambda} (k_1) \rangle^{-s-1} \lesssim 
\langle k_1 \rangle^{-2s-4}$, we use the H\"{o}lder inequality and 
the Young inequality to have 
\begin{align*}
& \| \langle k \rangle^{s+1} \langle \tau-p_{\lambda}(k ) \rangle^{-1/4}
\widehat{u}* \widehat{v} \|_{l_k^2 L_{\tau}^2} 
\lesssim  \| \langle k \rangle^{s+3/2} \langle \tau-p_{\lambda}(k) 
\rangle^{1/4} \widehat{u}* \widehat{v} \|_{l_k^{\infty} L_{\tau}^{\infty}} \\
& \hspace{0.3cm} \lesssim
 \| (\langle k \rangle^{s+11/4} \widehat{u})* \widehat{v} \|_{l_k^{\infty} L_{\tau}^{\infty} }
 \lesssim \| \langle k \rangle^{-3s-21/4} \|_{l_k^{\infty}} \| u \|_{X^{-3s-1,s+1}} \| v \|_{X^{-3s-1,s+1}},  
\end{align*}
which is an appropriate bound. 

\vspace{0.3em}

From the similar argument to above,  we obtain the desired estimate for $\Omega_{12c}$.

\vspace{0.5em}

\noindent
\underline{\bf Estimate for $\Omega_2$}\\

\noindent
(II) We divide $\Omega_{2}$ into three parts as follows;
\begin{align*}
\Omega_{21} & := \bigl\{ (\tau,k,\tau_1,k_1) \in \Omega_2~;~ 
L_{\max}= |\tau-p_{\lambda}(k)| \bigr\}, \\
\Omega_{22} & := \bigl\{ (\tau,k,\tau_1,k_1) \in \Omega_2~;~ 
L_{\max}= |\tau_1-p_{\lambda}(k_1)| \bigr\}, \\
\Omega_{23} & := \bigl\{ (\tau,k,\tau_1,k_1) \in \Omega_2~;~ 
L_{\max}= |(\tau-\tau_1)-p_{\lambda}(k-k_1)| \bigr\}.
\end{align*}
We omit the estimate for $\Omega_{23}$ because this case is identical to 
$\Omega_{22}$. Note that $\widehat{u}* \widehat{v}$ is supported on 
$D_3$ in $\Omega_2$. 
When $|k| \lesssim |k_1|^{-4}$, H\"{o}lder's and Young's inequalities show 
\begin{align*}
 \| |k| \langle \tau-p_{\lambda}(k) \rangle^{-1/4} \widehat{u} * \widehat{v} \|_{l_k^2 L_{\tau}^2}
& \lesssim \| (\langle k\rangle^{-3} \widehat{u})* (\langle k \rangle^{-3} \widehat{v}) \|_{l_k^{\infty}L_{\tau}^2} \\
&  \lesssim \| u \|_{X^{-3,0}} \| v \|_{Y^{-3}},
\end{align*}
which implies the desired estimate. So we only deal with the case 
$|k_1|^{-4} \lesssim |k| \leq 1$. 

\vspace{0.3em}

\noindent
(IIa) We prove the estimate for $\Omega_{21}$. We use the H\"{o}lder inequality and the Young inequality to obtain
\begin{align*}
& \| |k| \langle \tau-p_{\lambda}(k) \rangle^{s/2} \widehat{u}* \widehat{v} \|_{l_k^2 L_{\tau}^2} 
\lesssim  \| |k|^{1+s/2}  (\langle k \rangle^s \widehat{u}) * (\langle k \rangle^{s} \widehat{v}) \|_{l_k^{2} L_{\tau}^2} \\
& \hspace{0.8cm} \lesssim \| (\langle k \rangle^{s} \widehat{u}) * (\langle k \rangle^s  \widehat{v} ) \|_{l_k^{\infty} L_{\tau}^2} \lesssim 
\| u \|_{X^{s,0}} \| v \|_{Y^s}.
\end{align*}
Next we estimate the $Y^s$ norm of $\Lambda^{-1} \p_x (uv)$. Combining H\"{o}lder's and Young's inequalities, we have   
\begin{align*}
\| |k| \langle \tau-p_{\lambda}(k) \rangle^{-1} \widehat{u} * \widehat{v}
\|_{l_k^2 L_{\tau}^1} 
\lesssim & \| (\langle k \rangle^{-2} \widehat{u}) * (\langle k \rangle^{-2} \widehat{v}) \|_{l_{k}^{\infty} L_{\tau}^1} \\
\lesssim & \| \langle k \rangle^{-2} \widehat{u} \|_{l_k^2 L_{\tau}^1}
\| \langle k \rangle^{-2} \widehat{v} \|_{l_k^2 L_{\tau}^1},
\end{align*}
which is an appropriate bound.

\noindent
(IIb) We consider the estimate for $\Omega_{22}$. Following $\| u \|_{Y^s} \lesssim \| u \|_{X^{s,1/2+}}$, 
it suffices to show 
\begin{align} \label{BE_3}
\| |k| \langle \tau-p_{\lambda}(k) \rangle^{-1/2+} \widehat{u}* 
\widehat{v} \|_{l_k^2 L_{\tau}^2} \lesssim \| u \|_{Z^s} \| v \|_{Z^s}
\end{align}
in $\Omega_{22}$. We consider three subregions as follows;
\begin{align*}
\Omega_{22a} & := \bigl\{ (\tau,k,\tau_1,k_1) \in \Omega_{22}~;~ 
|k k_1|^4 \lesssim |\tau_1-p_{\lambda}(k_1)| \lesssim |k_1|^4 \bigr\}, \\
\Omega_{22b} & := \bigl\{ (\tau,k,\tau_1,k_1) \in \Omega_{22}~;~ 
|k_1^4| \lesssim |\tau_1-p_{\lambda} (k_1)| \lesssim |k_1|^5 \bigr\}, \\
\Omega_{22c} & := \bigl\{ (\tau,k,\tau_1,k_1) \in \Omega_{22}~;~ 
|k_1|^5 \lesssim |\tau_1-p_{\lambda}(k_1)|  \bigr\}.
\end{align*}
In $\Omega_{22a}$, $\widehat{u}$ is restricted to $D_1$. We use the H\"{o}lder inequality and the Young inequality to obtain 
\begin{align*}
& \| |k| \langle \tau-p_{\lambda}(k) \rangle^{-1/2+} \widehat{u}* \widehat{v} \|_{l_k^2 L_{\tau}^2}
\lesssim \| |k|^{1/4} 
(\langle k \rangle^s \langle \tau-p_{\lambda}(k) \rangle^{3/4} \widehat{u} )
* (\langle k \rangle^{-s-3} \widehat{v}) \|_{l_k^2 L_{\tau}^2} \\
& \hspace{0.3cm} \lesssim \| (\langle k \rangle^{s} 
\langle \tau-p_{\lambda}(k) \rangle^{3/4} \widehat{u}) * (\langle k \rangle^s 
\widehat{v})  \|_{l_k^{\infty} L_{\tau}^2 } 
\lesssim \| u \|_{X^{s,3/4} } \| v \|_{Y^s}.
\end{align*}

We consider the estimate for $\Omega_{22b}$ and $\Omega_{22c}$. From the estimate for $\Omega_{12}$, 
we immediately obtain (\ref{BE_3}) in the case 
$|\tau-p_{\lambda} (k)|  \sim |\tau_1-p_{\lambda} (k_1) |$. So we only deal with the case 
$|\tau_1-p_{\lambda}(k_1)| \sim |(\tau-\tau_1) -p_{\lambda}(k-k_1)|$. 
In $\Omega_{22b}$, both $\widehat{u}$ and $\widehat{v}$ are restricted to $D_2$. 
We use Lemma~\ref{lem_L_4} with $b'=1/2-$ and $b=s/2+1$ to have 
\begin{align*}
\| |k| \langle \tau-p_{\lambda}(k) \rangle^{-1/2+} \widehat{u}* \widehat{v} \|_{l_k^2 L_{\tau}^2} 
\lesssim & \| (J^{-3s-1} \Lambda^{s+1} u) (J^{-2s-4} v ) \|_{X^{0, -1/2+}} \\
& \lesssim \| u \|_{X^{-3s-1,s+1}} \| v \|_{X^{-2s-4,s/2+1}},
\end{align*}
which shows the desired estimate since $\| v \|_{X^{-2s-4,s/2+1}} \lesssim 
\| v \|_{X^{-3s-1,s+1}}$ in $D_{2}$ for $s \geq -3/2$. 

The case $\Omega_{22c}$ is almost identical to the above case.

\vspace{0.5em}

\noindent
\underline{\bf Estimate for $\Omega_{3}$} \\

\noindent
(III) From the algebraic relation (\ref{alg}), $L_{\max} \gtrsim |k_1 k^4|$. 
We decompose $\Omega_3$ into three parts as follows;
\begin{align*}
\Omega_{31} & := \bigl\{ (\tau,k,\tau_1,k_1) \in \Omega_3~;~ 
L_{\max}= |\tau_1-p_{\lambda}(k_1)| \bigr\}, \\
\Omega_{32} & := \bigl\{ (\tau,k,\tau_1,k_1) \in \Omega_3~;~ 
L_{\max}= |\tau-p_{\lambda}(k)| \bigr\}, \\
\Omega_{33} & := \bigl\{ (\tau,k,\tau_1,k_1) \in \Omega_3~;~ 
L_{\max}= |\tau_2-p_{\lambda}(k_2) | \bigr\}.
\end{align*}

\noindent
(IIIa) Firstly, we consider the case $\widehat{u} * \widehat{v}$ is supported on $D_3$. In this case, 
either $|\tau-p_{\lambda}(k)| \sim |\tau_1-p_{\lambda} (k_1) | \gtrsim |k|^5$ or 
$|\tau-p_{\lambda} (k)| \sim |\tau_2 -p_{\lambda} (k_2)| \gtrsim |k|^5$ holds. 
In the former case, $\widehat{u}$ are supported on $D_3$. We use the Young inequality to obtain 
\begin{align*}
\| \langle k \rangle^{-s/2} \langle \tau-p_{\lambda}(k) \rangle^{s/2} \widehat{u} * \widehat{v} \|_{l_k^2 L_{\tau}^2}
& \lesssim \| (J^{-s/2-1} \Lambda^{s/2+1} u) (J^{-4} v) \|_{L_{t,x}^2} \\
&  \lesssim \| u \|_{X^{-s/2-1,s/2+1}} \| \langle k \rangle^{-4} v \|_{l_k^{1} L_{\tau}^1},  
\end{align*}
which is bounded by $\| u \|_{X^{-s/2-1,s/2+1}} \| v \|_{Y^s}$ from the Schwarz inequality. 
The latter case is almost identical to the above case.  

Secondly, we deal with the case $\widehat{v}$ is supported on $D_3$. 
From (\ref{alg}), $|\tau_2 -p_{\lambda} (k_2)| \sim |\tau-p_{\lambda} (k)| \gtrsim |k|^5$ or 
$|\tau_2-p_{\lambda} (k_2) | \sim |\tau_1-p_{\lambda} (k_1)| \gtrsim |k|^5$ holds. 
In the former case, we have already proven (\ref{BE_T2}). So we consider the latter case.  
We may assume that $\widehat{u}$ is restricted to $D_3$ 
and $|\tau-p_{\lambda} (k)| \lesssim |k|^5$. We use the H\"{o}lder inequality and the Young inequality to have 
\begin{align*}
& \| \langle k \rangle^{s+1} \langle \tau-p_{\lambda}(k) \rangle^{-1/4} \widehat{u} * \widehat{v} \|_{l_k^2 L_{\tau}^2}
\lesssim \| \langle k \rangle^{s+3/2} \langle \tau-p_{\lambda} (k) \rangle^{1/4} \widehat{u}* 
\widehat{v} \|_{l_k^{\infty} L_{\tau}^{\infty}} \\
& ~~ \lesssim \| \langle k \rangle^{-3s-21/4} \|_{l_k^{\infty}} 
\| u \|_{X^{-s/2-1,s/2+1}} \| v \|_{X^{-s/2-1,s/2+1}}, 
\end{align*}
which shows the required estimate. Therefore we only deal with the case 
both $\widehat{u}* \widehat{v}$ and $\widehat{v}$ are supported on $D_1 \cup D_2$. 

\vspace{0.3em}
\noindent
(IIIb) We estimate (\ref{BE_T2}) for $\Omega_{31}$. In $\Omega_{31}$, $\widehat{u}$ is supported on $D_3$ from 
$L_{\max} \gtrsim |k_1 k^4|$. 
We use Lemma~\ref{lem_L_4} with $b'=1/4$ and $b=1/2$ to obtain 
\begin{align*}
\| \langle k \rangle^{s+1} \langle \tau-p_{\lambda}(k) \rangle^{-1/4} \widehat{u} * \widehat{v} \|_{l_k^2 L_{\tau}^2}
& \lesssim \| (J^{-s/2-1} \Lambda^{s/2+1} u) (J^{-s-3} v) \|_{X^{0,-1/4}} \\
&  \lesssim \| u \|_{X^{-s/2-1,s/2+1}} \| v \|_{X^{s,1/2}},
\end{align*}
which is an appropriate bound. 

\vspace{0.3em}

\noindent
(IIIc) 
We consider the estimate for $\Omega_{32}$. From (\ref{alg}), 
$\widehat{u}* \widehat{v}$ is supported on $D_2$. 
We use Lemma~\ref{lem_L_4} with $b=1/4$ and $b'=1/2$ to have 
\begin{align*} 
\| \langle k \rangle^{-3s} \langle \tau-p_{\lambda}(k) \rangle^{s} 
\widehat{u} * \widehat{v} \|_{l_k^2 L_{\tau}^2} 
& \lesssim 
\| (J^s u) (J^s v) \|_{L_{t,x}^2} \\
& \lesssim \| u \|_{X^{s,1/4}} \| v \|_{X^{s,1/2}},
\end{align*}
which is an appropriate bound.

Next, we estimate the $Y^s$ norm of $\Lambda^{-1} \p_{x} (uv)$. 
The Young inequality shows 
\begin{align*}
 \| \langle k \rangle^{s+1} \langle \tau-p_{\lambda}(k) \rangle^{-1} \widehat{u} * \widehat{v} \|_{l_k^2 L_{\tau}^1}
& \lesssim \| (\langle k \rangle^{-4}  \widehat{u})* (\langle k \rangle^{s} \widehat{v}) \|_{l_k^2 L_{\tau}^1} \\
&  \lesssim \| \langle k \rangle^{-4} \widehat{u} \|_{l_k^1 L_{\tau}^2 } \| v \|_{Y^s},
\end{align*}
which implies the desired estimate from the Schwarz inequality. 

\vspace{0.3em}

\noindent
(IIId) We consider the estimate for $\Omega_{33}$. 
From (\ref{alg}), we may assume that $\widehat{v}$ is supported on $D_2$ and 
$|\tau-p_{\lambda} (k)| \lesssim |k|^5$.  
In the case $\widehat{u}$ is supported on $D_1 \cup D_2$, 
we use Lemma~\ref{lem_L_4} with $b'=1/4$ and $b=1/2$ to obtain
\begin{align*}
\| \langle k \rangle^{s+1} \langle \tau-p_{\lambda} (k) \rangle^{-1/4}
\widehat{u} * \widehat{v} \|_{l_k^2 L_{\tau}^2} 
& \lesssim \| (J^{-s-3} u) (J^{-3s-1} \Lambda^{s+1} v ) \|_{X^{0,-1/4}} \\
& \lesssim \| u \|_{X^{s,1/2}} \| v \|_{X^{-3s-1,s+1}}.
\end{align*}
On the other hand, we consider the case $\widehat{u}$ is supported on $D_3$. 
Then we use Lemma~\ref{lem_L_4} with $b'=-s/2-1/4$ and $b=s/2+1$ to have 
\begin{align*}
& \| \langle k \rangle^{s+1} \langle \tau-p_{\lambda} (k) \rangle^{-1/4}
\widehat{u} * \widehat{v} \|_{l_k^2 L_{\tau}^2} 
\lesssim \| \langle k \rangle^{-3s/2-3/2} 
\langle \tau-p_{\lambda} (k) \rangle^{s/2+1/4}
\widehat{u} * \widehat{v} \|_{l_k^2 L_{\tau}^2} \\
& \lesssim 
 \| (J^{-7s/2-11/2} u) (J^{-3s-1 } \Lambda^{s+1} v) \|_{X^{0,s/2+1/2}}
\lesssim \| u \|_{X^{-7s/2-11/2, s/2+1}} \| v \|_{X^{-3s-1,s+1}},
\end{align*}
which shows the desired estimate since $-7s/2 -11/2 \leq -s/2-1$ for $s \geq -3/2$.

\vspace{0.5em}

\noindent
\underline{\bf Estimate for $\Omega_{4}$}\\

\noindent
(VI) In $\Omega_4$, $\widehat{u}$ is restricted to $D_3$. 
We divide $\Omega_4$ into three parts as follows;
\begin{align*}
\Omega_{41}:= & \bigl\{ (\tau,k, \tau_1, k_1) \in \Omega_{4}~; ~
L_{\max}=  |\tau_1-p_{\lambda} (k_1)|  \bigr\}, \\
\Omega_{42}:= & \bigl\{ (\tau,k, \tau_1, k_1) \in \Omega_{4}~; ~
L_{\max}=  |\tau-p_{\lambda} (k)|  \bigr\}, \\
\Omega_{43}:= & \bigl\{ (\tau,k, \tau_1, k_1) \in \Omega_{4}~; ~
L_{\max}=  |\tau_2-p_{\lambda} (k_2)|  \bigr\}.
\end{align*}
When $|k_1| \lesssim |k|^{-4}$, we easily obtain the desired estimate combining H\"{o}lder's and Young's inequalities. 
So we only deal with the case $|k|^{-4} \lesssim |k_1| \leq 1$.  

\vspace{0.3em}

\noindent
(VIa) In $\Omega_{41}$, we use the H\"{o}lder inequality and the Young inequality to obtain 
\begin{align*}
& \| \langle k \rangle^{s+1} \langle \tau-p_{\lambda}(k) \rangle^{-1/4} \widehat{u}* \widehat{v} \|_{l_k^2 L_{\tau}^2} 
\lesssim \| (|k|^{-1/4} \langle \tau-p_{\lambda}(k) \rangle^{1/4} \widehat{u} ) * ( \langle k \rangle^{s} \widehat{v} ) \|_{l_k^2 L_{\tau}^2} \\
&~~\lesssim \| |k|^{-1/4} \langle \tau-p_{\lambda}(k) \rangle^{1/4} \widehat{u} \|_{l_k^{1} L_{\tau}^2} \| v \|_{Y^s}
 \lesssim  \| u \|_{X^{0,1/4}} \| v \|_{Y^s}. 
\end{align*}

\noindent
(VIb) In $\Omega_{42}$, we use Young's inequality to have 
\begin{align*}
\| \langle k \rangle^{s+1} \langle \tau-p_{\lambda}(k) \rangle^{-1/4} \widehat{u} * \widehat{v} \|_{l_k^2 L_{\tau}^2}
& \lesssim \| (|k|^{-1/4}  \widehat{u})* (\langle k \rangle^{s} \widehat{v}) \|_{l_{k}^2 L_{\tau}^2} \\
& \lesssim \| |k|^{-1/4} \widehat{u} \|_{l_k^1 L_{\tau}^2} \| v \|_{Y^{s}}, 
\end{align*}
which is an appropriate bound from Schwarz's inequality.  

\noindent
(VIc) From $\langle k_2 \rangle^{-s} \langle \tau_2-p_{\lambda}(k_2) \rangle^{-1/4} \lesssim |k_1|^{-1/4} \langle k_2 \rangle^{-s-1}$ in $\Omega_{43}$, 
we use the H\"{o}lder inequality and the Young inequality to have 
\begin{align*}
& \| \langle k \rangle^{s+1} \langle \tau-p_{\lambda}(k) \rangle^{-1/4} \widehat{u} * \widehat{v} \|_{l_k^2 L_{\tau}^2} 
\lesssim  \| ( |k|^{-1/4} \widehat{u} ) * (\langle k \rangle^s  \langle \tau-p_{\lambda}(k) \rangle^{1/4} \widehat{v}) \|_{l_k^{2} L_{\tau}^2} \\
& \hspace{0.3cm} \lesssim  \| |k|^{-1/4} \widehat{u} \|_{l_k^1 L_{\tau}^1} 
\| v \|_{X^{s,1/4}} \lesssim \| u \|_{Y^s} \| v \|_{X^{s,1/4}}.
\end{align*}

\vspace{0.5em}

\noindent
\underline{\bf Estimate for $\Omega_{5}$}\\ 

We decompose $\Omega_{5}$ into three parts as follows;
\begin{align*}
\Omega_{51}:= & \bigl\{ (\tau,k ,\tau_1, k_1) \in \Omega_{5} ~: ~ 
L_{\max}= |\tau-p_{\lambda}(k)| \bigr\}, \\
\Omega_{52}:= & \bigl\{ (\tau,k ,\tau_1, k_1) \in \Omega_{5} ~: ~ 
L_{\max}= |\tau_1-p_{\lambda}(k_1)| \bigr\}, \\
\Omega_{53}:= & \bigl\{ (\tau,k ,\tau_1, k_1) \in \Omega_{5} ~: ~ 
L_{\max}= |\tau_2-p_{\lambda}(k_2)| \bigr\}.
\end{align*}

\vspace{0.3em}

\noindent
(Va) In $\Omega_{51}$, $\widehat{u} * \widehat{v}$ is supported on $D_3$. We divide this region into 
\begin{align*}
\Omega_{51a}:= & \bigl\{ (\tau,k ,\tau_1, k_1) \in \Omega_{51} ~: ~ 
 |\tau-p_{\lambda}(k)| \sim |k^5| \bigr\}, \\
\Omega_{51b}:= & \Omega_{51} \setminus \Omega_{51a}. 
\end{align*}
In $\Omega_{51a}$, both $\widehat{u}$ and $\widehat{v}$ are supported on $D_1 \cup D_2$ from (\ref{alg}). 
We use Lemma~\ref{lem_L_4} with $b'=1/4$ and $b=1/2$ to have 
\begin{align*}
\| \langle k \rangle^{-s/2} \langle \tau-p_{\lambda}(k) \rangle^{s/2} 
\widehat{u}* \widehat{v} \|_{l_k^2 L_{\tau}^2} 
& \lesssim  \|(J^s u) (J^s v ) \|_{L_{t,x}^2} \\
& \lesssim \| u \|_{X^{s,1/2}} \| v \|_{X^{s,1/4}}. 
\end{align*}

In $\Omega_{51b}$, either $|\tau-p_{\lambda} (k)| \sim |\tau_1-p_{\lambda} (k_1)| $ or $|\tau-p_{\lambda} (k)| \sim |\tau_2-p_{\lambda} (k_2)|$ holds. 
Following the similar argument to the case $\Omega_{11}$, we obtain the desired estimate in $\Omega_{51b}$. 

\vspace{0.5em}

\noindent
(Vb) We consider the estimate for $\Omega_{52}$. From (\ref{alg}), $\widehat{u}$ is supported on $D_3$. 
We divide $\Omega_{52}$ into 
\begin{align*}
\Omega_{52a}:= & \bigl\{ (\tau,k ,\tau_1, k_1) \in \Omega_{52} ~: ~ 
 |\tau_1-p_{\lambda}(k_1)| \sim |k_1^5| \bigr\}, \\
\Omega_{52b}:= & \Omega_{52} \setminus \Omega_{52a}. 
\end{align*}
In $\Omega_{52a}$, $\widehat{u} * \widehat{v}$ and $\widehat{v}$ are supported on $D_1 \cup D_2$ under this assumption. 
Then we use Lemma~\ref{lem_L_4} with $b'=1/4$ and $b=1/2$ to have 
\begin{align*}
\| \langle k \rangle^{s+1} \langle \tau-p_{\lambda}(k) \rangle^{-1/4} 
\widehat{u}* \widehat{v} \|_{l_k^2 L_{\tau}^2} 
& \lesssim  \|(J^{-s/2-1} \Lambda^{s/2+1} u) (J^{-s-3} v ) \|_{X^{0,-1/4}} \\
& \lesssim \| u \|_{X^{-s/2-1,s/2+1}} \| v \|_{X^{s,1/2}}, 
\end{align*}
which is an appropriate bound. 

In $\Omega_{52b}$, either $|\tau_1-p_{\lambda} (k_1)| \sim |\tau-p_{\lambda} (k)|$ or $|\tau_1 -p_{\lambda}(k_1) | \sim |\tau_2-p_{\lambda}(k_2) |$ holds. 
These cases are almost identical to the case (IIIa). 

In the same manner as above, we obtain the desired estimate in $\Omega_{53}$ 
by symmetry. 
\end{proof}

\section{Proof of the local well-posedness}
In this section, we give the proof of Theorem~\ref{LWP_TK} by the 
iteration method. 
Here we put $U(t):=\mathcal{F}_{k}^{-1} \exp(i p(k) t) \mathcal{F}_x$ and $p(k):=k^5 +\beta k^3$. 
We obtain the local well-posedness result in the following sense.
\begin{prop} \label{prop_well}
Let $-3/2 \leq s \leq -1$ and $r>1$. For any $u_0 \in B_r(\dot{H}^{s})$, there exist $T=T(r)>0$ and a unique solution 
$u \in Z^{s}([0,T])$ satisfying the following integral form for (\ref{KT_1});
\begin{align} \label{integral-1}
u(t)=  U(t) u_0 -\int_{0}^t U(t-s) \p_x (u (s))^2 ds.
\end{align}
Moreover the data-to-solution map, 
$B_r(\dot{H}^{s}) \ni u_0 \mapsto u \in  Z^{s} ([0,T])$, is locally Lipschitz continuous. 
\end{prop}
\begin{proof}
We first prove the existence of the solution by the fixed point argument. Here $\lambda$ is a 
sufficiently large number determined later. 
For any $u_0 \in B_r (\dot{H}^s)$, from (\ref{in_sm}),  
$\| u_{0,\lambda} \|_{\dot{H}^s} \leq \lambda^{-2} r $ when $-3/2 \leq s <0$. 
Therefore we prove, for any $u_{0,\lambda} \in B_{\lambda^{-2} r}(\dot{H}^s)$, there exists 
$u_{\lambda} \in Z^s([0,1])$ satisfying 
\begin{align} \label{integral-2}
M[u_{\lambda} ] (t) =u_{\lambda} (t), \hspace{0.3em} 
M[u_{\lambda}] (t)=U_{\lambda} (t) u_{0,\lambda}-\int_0^t 
U_{\lambda} (t-s) \p_x (u_{\lambda}(s))^2 ds. 
\end{align}
Following Propositions~\ref{prop_BE_2} and \ref{prop_linear2}, 
we obtain the bilinear estimate as follows;
\begin{align} \label{BE_time}
\Bigl\| \int_0^t U_{\lambda} (t-s) \p_x (u_{\lambda} (s) v_{\lambda} (s)) ds \Bigr\|_{Z^s([0,1])} 
\leq C_1 \| u_{\lambda} \|_{Z^s([0,1])} \| v_{\lambda} \|_{Z^s([0,1])},
\end{align}
for some constant $C_1>0$. 
From Proposition~\ref{prop_linear1} and (\ref{BE_time}), we have 
\begin{align*}
\| M [u_{\lambda}] \|_{Z^s ([0,1])} \leq C_1 \bigl( \| u_{0,\lambda} \|_{\dot{H}^s}+
\| u_{\lambda} \|_{Z^s([0,1])}^2 \bigr). 
\end{align*}
Here we choose $\lambda^2 \geq 8 C_1^2 r$ so that 
$M$ is a map from $B_{2C_1 \lambda^{-2} r} (Z^s ([0,1]))$ to itself. 
In the same manner as above, we obtain
 \begin{align*}
 \| M [u_{\lambda}]  -M [v_{\lambda}] \|_{Z^s ([0,1])} & \leq C_1  
\| u_{\lambda} +  v_{\lambda} \|_{Z^s ([0,1])} 
\| u_{\lambda}-v_{\lambda} \|_{Z^s ([0,1])}  \\
 & \leq 4 \lambda^{-2} C_1^2 r \| u_{\lambda} -v_{\lambda} \|_{Z^s ([0,1])} 
 \leq \frac{1}{2} \| u_{\lambda} -v_{\lambda} \|_{Z^s ([0,1])}, 
\end{align*}
which implies that $M$ is a contraction map on $B_{2 C_1 \lambda^{-2} r} (Z^s([0,1]))$. 
From the fixed point argument, 
we construct the solution to (\ref{integral-2}) on $[0,1]$. Here we put 
$u(t,x):= \lambda^{4} u_{\lambda} (\lambda^5 t, \lambda x)$. Then $u$ solves (\ref{integral-1}) 
on $[0,T]$ where the lifetime $T$ satisfies $T \sim \lambda^{-5} \sim r^{-5/2}$. Moreover, following the standard 
argument, we show that the data-to-solution map is locally Lipschitz continuous. 

Moreover uniqueness can be extended to the whole $Z^s([0,T])$. 
This proof is based on Muramatu and Taoka's work \cite{MT}. 
For the details, see \cite{TK_K}.
\end{proof}


\section{Proof of the global well-posedness}

 In this section, we extend the local solution obtained above globally in time by the I-method. 
When $s \geq -1$, 
Hirayama \cite{Hi} obtained LWP for (\ref{KT_1}) in the function space $W^s([0,T])$ equipped with the norm 
\begin{align*}
\| u \|_{W^s([0,T])}:= \| u \|_{X^{s,1/2}([0,T])}+ \| u \|_{Y^{s}([0,T])}.
\end{align*}
These local-in-time solutions are shown to exist on an arbitrary time interval for $0> s \geq -1$.   
Note that $s=-1$ is optimal in such sense that the bilinear estimate in the standard Bourgain space 
fails for $s<-1$.  
The proof is an adaptation of the argument presented for the periodic KdV equation in 
\cite{CoKeSt}. Remark that we encounter difficulty such that the Kawahara equation 
has less symmetries than the KdV equation. 
Before modified energies are introduced, we prepare some notations. 
A $l$ multiplier is a function $M; \mathbb{R}^l \rightarrow \mathbb{C}$.  
We say a $l$ multiplier $M$ is symmetric if 
$M(k_1, k_2, \cdots, k_l)=M(k_{\sigma(1)} , k_{\sigma(2) }, \cdots, k_{\sigma(l)} )$ for 
all $\sigma \in S_l$. 
The symmetrization of a $l$ multiplier $M$ is defined by 
\begin{align*}
[M]_{sym} (k_1, k_2, \cdots, k_l) := \frac{1}{ l! } \sum_{\sigma \in S_l}
M (k_{\sigma(1)}, k_{\sigma(2)}, \cdots, k_{\sigma(l)} ).
\end{align*}
We define a $l$-linear functional associated to the function $M$ acting on $l$ functions $u_1,u_2, \cdots, u_l$, 
\begin{align*}
\Lambda_l (M; u_1, u_2, \cdots, u_l):= \int_{k_1+k_2+\cdots +k_l =0 } M(k_1,k_2, \cdots, k_l) 
\prod_{i=1}^{l} \widehat{u}_i(k_i). 
\end{align*}  
$ \Lambda_l (M; u, \cdots,u  )$ is simply written as $\Lambda_l (M)$. 
We recall the original modified energy $E_I^{(2)} (u) (t) = \| I u (t) \|_{L^2}^2$. 
We use this functional to obtain GWP for $-21/26< s <0$ but not 
 $-1 \leq s \leq -21/26$. Then we construct new modified energies by adding some correction terms to 
 $E_I^{(2)}(u)$, following the argument to \cite{CoKeSt}. 
Using $u$ is real valued and $m$ is even, we use the Plancherel theorem to have
\begin{align*}
E_I^{(2)} (u) (t)= \Lambda_2 ( m(k_1) m(k_2)) (t).
\end{align*} 
Here $a_l$, $b_l$ denote $a_l =i \sum_{j=1}^{l} k_j^5$ and $b_l= i \sum_{j=1}^{l} k_j^{3}$. 
We compute the time derivative of the modified energy $E_I^{(2)} (u)$ to have 
\begin{align*}
\frac{d}{dt} E_{I}^{(2)} (u) (t)= & \Lambda_{2} ((a_2+ \lambda^{-2} \beta b_2) m(\xi_1) m(\xi_2) ) (t) \\
- & 2i \Lambda_3 ([ ( k_2+ k_3) m(k_1) m(k_2+ k_3)  ]_{sym}) (t) . 
\end{align*}
Here the first term vanishes because $a_2=0$ and $b_2=0$. Therefore the time derivative of $E_I^{(2)}(u)$ has the cubic form as follows; 
\begin{align*}
\frac{d}{ dt} E_I^{(2)} (u) (t)= \Lambda_3 (M_3) (t), \hspace{0.3cm}
M_3(k_1, k_2,k_3) = -2i [ m(k_1) m(k_{23}) k_{23}]_{sym}, 
\end{align*}
where $k_{ij}= k_i+ k_j$ for $i \neq j$. 
We add a correction term $\Lambda_3( \sigma_3)$ to the modified energy $E_I^{(2)} (u)$ to construct a new modified energy $E_I^{(3)} (u)$. 
Namely, 
\begin{align*} 
E_I^{(3)} (u) (t) =E_I^{(2)} (u)(t)+\Lambda_3 (\sigma_3) (t),
\end{align*} 
where the symmetric function $\sigma_3$ is determined later. Similarly, the time derivative of $E_I^{(3)} (u)$ is expressed by
\begin{align*}
\frac{d}{dt} E_I^{(3)} (u) (t) = & \Lambda_3( M_3  ) (t) +\Lambda_3 ( (a_3+ \lambda^{-2} \beta b_3    ) \sigma_3) (t) \\
 - & 3i \Lambda_4 ([ \sigma_3 (k_1, k_2, k_{34}) k_{34} ]_{sym} ) (t).
\end{align*}
Here we choose $\sigma_3= - M_3 /( a_3+ \lambda^{-2} \beta b_3 )$ to cancel the cubic terms. 
Then, 
\begin{align*}
\frac{d}{ d t} E_I^{(3)} (u)(t)=\Lambda_4 (M_4) (t), \hspace{0.3cm}
M_4 (k_1,k_2,k_3,k_4)
:=-  3i \Lambda_4 ([ \sigma_3 (k_1, k_2, k_{34}) k_{34} ]_{sym} ).
\end{align*}
In the same manner, we define the third modified energy as follows;
\begin{align*}
E_I^{(4)} (u) (t):= E_I^{(3)} (u)(t)+  \Lambda_4(\sigma_4) (t), \hspace{0.3em}
\sigma_4:=- M_4/(a_4+\lambda^{-2} \beta b_4). 
\end{align*} 
Then we have
\begin{align*}
\frac{d}{ d t} E_I^{(4)} (u) (t):= & \Lambda_5 (M_5) (t), \\
M_5 (k_1, k_2, k_3,k_4, k_5):= & -4 i [ \sigma_4 (k_1,k_2,k_3, k_{45}) k_{45} ]_{sym}.
\end{align*}
Chen and Guo \cite{CG} obtained the upper bound of $M_4$ as follows.  
\begin{lem} \label{lem_M_4} 
Let $|k_1| \geq |k_2| \geq |k_3| \geq |k_4|$. Then we have 
\begin{align} \label{M_4}
|M_4 (k_1, k_2, k_3, k_4)| \lesssim \frac{ |a_4+ \beta \lambda^{-2} b_4| m( k_4^{*}) }
{ (N+|k_1|)^2 ( N+|k_2|)^2  (N+|k_3|)^3 (N+|k_4|)  }  
\end{align}
where $k_4^{*}:=\min \{ |k_l|, |k_{ij}| \} $.
\end{lem}
To establish this upper bound for the Kawahara equation is difficult 
because this equation has less symmetries than the KdV equation. 
Combining the bilinear Strichartz estimate (\ref{es_L_4_1}) and this upper bound (\ref{M_4}), 
we establish the following almost conservation law which 
controls the increment of the modified energy $E_I^{(4)} (u)$ in time. 
\begin{prop} \label{prop_ACL}
Let $ 0>s \geq -1$ and $N \gg 1$. Then there exists $C_1>0$ such that 
\begin{align} \label{ACL1}
\bigl| E_I^{(4)}(u)(t)-E_I^{(4)}(u)(t_0)  \bigr| \leq C_1 N^{5s}
\| I u (t_0) \|_{W^{0} ([t_0-1,t_0+1])}^5,
\end{align}
for any $t_0 \in \mathbb{R}$ and $t \in [t_0-1,t_0+1]$. 
\end{prop}

\begin{proof}
We may assume $t_0=0$ and $ \widehat{u} $ is non-negative. Since 
\begin{align*}
| E_I^{(4)} (u) (t) -E_I^{(4)} (u) (0) | \lesssim \int_{-1}^{1} \Lambda(M_5 ) (t) dt, 
\end{align*}
for any $t \in [-1,1]$, it suffices to show that
\begin{align} \label{ACL2}
\int_{-1}^1 \Lambda_5 \Bigl(\frac{M_5 (k_1,k_2,k_3,k_4,k_5)}{ m(k_1) m(k_2) m(k_3) m(k_4) m(k_5)} \Bigr)
(t) dt \lesssim N^{5s} \| u \|_{W^0([-1,1])}^5.
\end{align}
We suppose that $|k_1| \geq |k_2| \geq |k_3| \geq |k_4| \geq |k_5|$ without loss of generality. 
$M_5$ vanishes when $|k_i| \ll N$ for any $i=1,2,3,4,5$. So we can assume 
$|k_1| \sim |k_2| \gtrsim N$. From the definition of $M_5$, we have
\begin{align*}
|M_5 (k_1,k_2,k_3,k_4,k_5)| \lesssim |\sigma_4(k_3,k_4,k_5,k_{12}) k_{12}|. 
\end{align*}
From $k_3+k_4+k_5+k_{12}=0$, we only consider two cases as follows;
\begin{align*}
D_1:=& \bigl\{ (\vec{\tau}, \vec{k}) \in \mathbb{R}^5 \times \dot{\mathbb{Z}}_{\lambda}^5 ~; ~
|k_3| \sim |k_{12}| \gtrsim |k_4| \geq |k_5| \text{ and } |k_3 | \sim |k_{12}| \gtrsim N \bigr\}, \\
D_2:=& \bigl\{ (\vec{\tau}, \vec{k}) \in \mathbb{R}^5 \times \dot{\mathbb{Z}}_{\lambda}^5 ~; ~
|k_3| \sim |k_4| \gg  \max \{ |k_{12}|, k_5| \} \text{ and } |k_3 | \sim |k_4| \gtrsim N \bigr\}.
\end{align*}
where $\vec{\tau}:=(\tau_1,\tau_2, \cdots ,\tau_5)$ and 
$\vec{k}:= (k_1, k_2 , \cdots, k_5)$. 

\vspace{0.3em}

\noindent
(I) Firstly, we prove (\ref{ACL2}) in $D_1$.
From (\ref{M_4}), we easily obtain the upper bound of $M_5$ as  follows; 
\begin{align*}
| M_5(k_1,k_2, k_3, k_4, k_5)| \lesssim \frac{|k_{12}| }{ (N+|k_3|)^4 (N+|k_4|)^3 (N+|k_5|)^1}.
\end{align*}
From $|k_{12}| \sim |k_3| \gtrsim N$, we substitute this estimate into (\ref{ACL2}) and use the dyadic decompositions 
to have 
\begin{align*}
(\text{L. H. S. of (\ref{ACL2}} ) ) & \lesssim 
N^{5s} \int_{-1}^1 \Lambda_5 \bigl( |k_{12}| \langle k_1 \rangle^{-s} 
\langle k_2 \rangle^{-s} \langle k_3 \rangle^{-s-4} \langle k_4 \rangle^{-s-3} \langle k_5 \rangle^{-s-1} \bigr) (t) dt \\
& \lesssim N^{5s} \sum_{N_1} \sum_{N_2 \sim N_1} \sum_{N_3 \leq N_2} \sum_{N_4 \leq N_3} \sum_{N_5 \leq N_4} \\ 
& \hspace{0.3cm} \times N_1^{-s} N_2^{-s} N_3^{-s-4} \langle N_4 \rangle^{-s-3} \langle N_5 \rangle^{-s-1} 
\prod_{i=1}^5 \| u_{N_i} \|_{L_x^1 L_{t \in [-1,1]}^1 }
\end{align*}
where $u_{N_i} := P_{\{ |k_i| \sim N_i \} } u$ for dyadic numbers $N_i$ with $i=1,2,3,4,5$. 
From the Schwarz inequality, (\ref{ACL2}) is reduced to two estimates as follows;
\begin{align} \label{BE_G}
N_1^{-s} N_2^{-s} \| |\p_x| u_{N_1} u_{N_2}  \|_{X^{s,-1/2} } & \lesssim N_1^{-s-1} N_2^{-s-1} \| u_{N_1} \|_{W^0}
\| u_{N_2} \|_{W^0},  \\
\label{TR_G}
N_3^{-s-4}  \langle N_4 \rangle^{-s-3} \langle N_5 \rangle^{-s-1} \| \prod_{i=3}^5 u_{N_i} \|_{X^{-s,1/2}}  & 
\lesssim N_3^{-2s-2} \langle N_4 \rangle^{-s-2} \prod_{i=3}^5 \| u_{N_i} \|_{W^0}. 
\end{align}
If these estimates hold, the left hand side of (\ref{ACL2}) is bounded by 
\begin{align*}
& N^{5s} \sum_{N_1} \sum_{N_2 \sim N_1} \sum_{N_3 \leq N_2} \sum_{N_4 \leq N_3} \sum_{N_5 \leq N_4} 
 N_1^{-s-1} N_2^{-s-1} N_3^{-2s-2} \langle N_4 \rangle^{-s-2} \\
& \hspace{0.8cm} \times
\| u_{N_1} \|_{W^{0} ([-1,1])} \| u_{N_2} \|_{W^{0}([-1,1]) } \| u \|_{W^{0} ([-1,1]) }^3 \\
& \hspace{0.5cm} \lesssim  
N^{5s} \sum_{N_1} \sum_{N_2 \sim N_1} 
N_1^{-4s-4} \| u_{N_1} \|_{W^{0} ([-1,1])} \| u_{N_2} \|_{W^0 ([-1,1])} \| u \|_{W^0 ([-1,1])}^3, 
\end{align*}
which shows the desired estimate for $-1 \leq s <0$.

The bilinear estimate (\ref{BE_G}) has been already proven by Hirayama \cite{Hi}. 
So we only prove the trilinear estimate (\ref{TR_G}). From the Plancherel theorem, 
we have the identity, 
\begin{align*}
\| u_{N_3} u_{N_4} u_{N_5} \|_{X^{-s,1/2}} = \bigl\| \langle k \rangle^{-s} \langle \tau-p_{\lambda} (k) \rangle^{1/2} 
\prod_{i=3}^5 \widehat{u}_{N_i} (\tau_i, k_i) \bigr\|_{l_k^2 L_{\tau}^2 }, 
\end{align*} 
where $k=k_3+k_4+k_5$ and $\tau=\tau_3+ \tau_4 +\tau_5$. From the definition, $|k| \sim |k_{12}| \sim N_3$ in this case. 

\vspace{0.3em} 

\noindent
(Ia) We first consider the case  
$\langle \tau-p_{\lambda} (k) \rangle \lesssim \langle \tau_i-p_{\lambda} (k_i) \rangle$ for 
some $i=3,4,5$. By symmetry, we may assume  
$\langle \tau-p_{\lambda} (k) \rangle \lesssim \langle \tau_3 -p_{\lambda} (k_3) \rangle$. 
It suffices show that 
\begin{align} \label{red_1}
& N_3^{-2s-4}  \langle N_4 \rangle^{-s-3} \langle N_5 \rangle^{-s-1} \|  u_{N_3} u_{N_4} u_{N_5} \|_{L_{t,x}^2} \nonumber \\
& \hspace{0.5cm} \lesssim N_{3}^{-2s-4} \langle N_4 \rangle^{-s-2} \langle N_5 \rangle^{-s-1} \| u_{N_3} \|_{X^{0,0} } \| u_{N_4} \|_{Y^0} \| u_{N_5} \|_{Y^0}. 
\end{align}
H\"{o}lder's and Young's inequalities imply 
\begin{align*}
 \| \prod_{i=3}^{5} \widehat{u}_{N_i} (\tau_i, k_i) \|_{l_k^2 L_{\tau}^2} 
\lesssim & \| \widehat{u}_{N_3} \|_{l_k^2 L_{\tau}^2 } \| \widehat{u}_{N_4} \|_{l_k^1 L_{\tau}^1} 
\| \widehat{u}_{N_5} \|_{l_k^1 L_{\tau}^1} \\ 
\lesssim & 
 N_4^{1/2} N_5^{1/2} 
\| u_{N_3} \|_{X^{0,0}} \| u_{N_4} \|_{Y^0} \| u_{N_5} \|_{Y^0}. 
\end{align*}
We insert this into the left hand side of (\ref{red_1}) to obtain the required estimate. 

\vspace{0.3em}

\noindent
(Ib) Next, we consider the case 
$ \langle \tau-p_{\lambda} (k)  \rangle \gg \langle \tau_i-p_{\lambda}(k_i) \rangle $ 
for all $i=3,4,5$. 
In this case, we use the algebraic relation to have 
\begin{align} \label{res_es}
|\tau-p_{\lambda} (k)| \sim |-p_{\lambda} (k) + 
p_{\lambda}(k_1)+p_{\lambda} (k_2) +p_{\lambda} (k_3) | \lesssim |k_3|^4 |k_4| 
\end{align}
We use (\ref{res_es}) and the H\"{o}lder inequality to obtain
\begin{align*}
& N_3^{-2s-4} \langle N_4 \rangle^{-s-3} \langle N_5 \rangle^{-s-1} 
\| u_{N_3} u_{N_4} u_{N_5} \|_{X^{0,1/2}} \\
& \hspace{0.8cm}  \lesssim  
N_3^{-2s-2} \langle N_4 \rangle^{-s-5/2} \langle N_5 \rangle^{-s-1} 
\| u_{N_3} u_{N_4} u_{N_5} \|_{L_{t,x}^2} \\
& \hspace{0.8cm} \lesssim  N_3^{-2s-2} \langle N_4 \rangle^{-s-5/2} \langle N_5 \rangle^{-s-1}
\| u_{N_3} u_{N_4} \|_{L_{t,x}^2} \| u_{N_5} \|_{L_{t,x}^{\infty}} . 
\end{align*}
When $|k_{34}| \geq 1$, from (\ref{es_L_4_1}) and the Sobolev inequality, the right hand side is bounded by 
\begin{align*}
N_3^{-2s-2} \langle N_4 \rangle^{-s-5/2} \langle N_5 \rangle^{-s-1/2} \| u_{N_3} \|_{X^{0,3/8} } \| u_{N_4} \|_{X^{0,3/8}} \| u_{N_5} \|_{Y^0},
\end{align*}
which shows the required estimate. 
On the other hand, we deal with the case $|k_{34}| \leq 1 $. Combining the H\"{o}lder inequality and the Young inequality, we have 
\begin{align*}
\| u_{N_3} u_{N_4} \|_{L_{t,x}^2} \| u_{N_5} \|_{L_{t,x}^{\infty}} & \lesssim  
N_5^{1/2} \| \widehat{u}_{N_3} * \widehat{u}_{N_4} \|_{l_k^{\infty} L_{\tau}^2} \| u_{N_5} \|_{Y^0} \\
& \lesssim 
N_5^{1/2} \| u_{N_3} \|_{L_{t,x}^2} \| u_{N_4} \|_{Y^0} \| u_{N_5} \|_{Y^0}. 
\end{align*}
From this, we immediately obtain the desired estimate.  

\vspace{0.5em}

\noindent
(II) Secondly, we prove (\ref{ACL2}) in $D_2$. 
In this case, we have the upper bound of $M_5$ as follows;  
\begin{align*}
| M_5(k_1,k_2, k_3, k_4, k_5)| \lesssim \frac{|k_{12}| }{ (N+|k_3|)^2 (N+|k_4|)^2 (N+|k_{12}|)^2 (N+|k_5|)^2 }.
\end{align*}
In the same manner as above, (\ref{ACL2}) is reduced to (\ref{BE_G}) and 
\begin{align} \label{TR_G2}
N_3^{-s-2} N_4^{-s-2} \langle N_5 \rangle^{-s-2} \bigl\| \prod_{i=3}^5 \widehat{u}_{N_i} (\tau_i,k_i) 
\bigr\|_{X^{-s-2,1/2} } 
\lesssim  N_3^{-2s-2} \langle N_5 \rangle^{-s-3/2} \prod_{i=3}^5 \| u_{N_i} \|_{W^0}.
\end{align}
We now show the trilinear estimate (\ref{TR_G2}). 

\vspace{0.3em}

\noindent
(IIa) We first consider the case $|\tau-p_{\lambda} (k)| \lesssim |\tau_3-p_{\lambda} (k_3) | $. 
We use H\"{o}lder's inequality and Young's inequality to have 
\begin{align*}
& \| \langle k \rangle^{-s-2} \prod_{i=3}^5 \widehat{u}_{N_i} \|_{l_k^2 L_{\tau}^2} \lesssim 
\| \prod_{i=3}^5 \widehat{u}_{N_i} \|_{l_k^{\infty} L_{\tau}^2} \\
& \hspace{0.3cm} \lesssim \| \widehat{u}_{N_3} \|_{l_k^2 L_{\tau}^2} 
\| \widehat{u}_{N_4} \|_{l_k^{2} L_{\tau}^1} \| \widehat{u}_{N_5} \|_{l_k^1 L_{\tau}^1} 
\lesssim N_5^{1/2} \| u_{N_3} \|_{L_{t,x}^2 } \| u_{N_4} \|_{Y^0} \| u_{N_5} \|_{Y^0}, 
\end{align*}
which implies the desired estimate. 

\vspace{0.3em}

\noindent
(IIb) Next, we consider the case $|\tau-p_{\lambda} (k)| \gg |\tau_i-p_{\lambda} (k_i)| $ for all $i=3,4,5$. 
In this case, the algebraic relation implies 
\begin{align*}
|\tau-p_{\lambda} (k) | \lesssim \max \{ |k|, |k_5| \} |k_3|^4. 
\end{align*}
We use the H\"{o}lder inequality and the Young inequality to obtain 
\begin{align*}
\text{(L. H. S. of (\ref{TR_G2}))} & \lesssim 
N_3^{-2s-2} \langle N_5 \rangle^{-s-3/2} \| \prod_{i=3}^5 \widehat{u}_{N_i} \|_{l_k^{\infty} L_{\tau}^2} \\
& \lesssim N_{3}^{-2s-2} \langle N_{5} \rangle^{-s-3/2} \| u_{N_3} u_{N_4} \|_{L_{t,x}^2} \| u_{N_5} \|_{Y^0},
\end{align*}
which is an appropriate bound from the above argument. 
\end{proof}

Next, we estimate the difference between the almost conserved quantity $E_I^{(4)} (u)$ and the first modified energy $E_{I}^{(2)} (u)$ when 
the time is fixed. We call this estimate the fixed time difference.
\begin{prop} \label{prop_FTD}
Let $0> s \geq -1$ and $N \gg 1$. Then there exists $C_2>0$ such that  
\begin{align} \label{FTD}
| E_I^{(4)} (u)(t_0) -E_I^{(2)} (u) (t_0) | \leq C_2 (\| I u(t_0) \|_{L_x^2}^3+ \| I u(t_0) \|_{L_x^2}^4 ), 
\end{align}
for any $t_0 \in \mathbb{R}$.
\end{prop}
 
\begin{proof}
From the definition of the modified energies, it suffices to show that 
\begin{align*}
| \Lambda_3 (\sigma_3) (t_0) | \lesssim \| I u(t_0) \|_{L^2}^3, \hspace{0.3cm}
|\Lambda_4 (\sigma_4) (t_0) | \lesssim \| I u(t_0) \|_{L^2}^4. 
\end{align*}
These estimates are reduced to the following estimates. 
\begin{align} \label{FTD_C}
\Bigl| \Lambda_3 \Bigl( \frac{M_3 (k_1,k_2,k_3) }{ (a_3+ \beta \lambda^{-2} b_3) m(k_1) m(k_2) m(k_3) } (t_0)  \Bigr)
\Bigr| & \lesssim \| u (t_0) \|_{L^2}^3, \\
\label{FTD_Q}
\Bigl| \Lambda_4 \Bigl( \frac{M_4 (k_1,k_2,k_3,k_4) }{ (a_4+ \beta \lambda^{-2} b_4) m(k_1) m(k_2) m(k_3) m(k_4) } (t_0)  \Bigr)
\Bigr| & \lesssim \| u (t_0) \|_{L^2}^4. 
\end{align}
Firstly, we prove (\ref{FTD_C}) when $|k_1| \geq |k_2| \geq |k_3|$. Following the mean value theorem, we easily obtain 
the upper bound of $M_3$ as follows; 
\begin{align*}
| M_3 (k_1, k_2,k_3)| \lesssim |k_3| m(k_3)^2.
\end{align*}
If $|k_i| \ll N$ for all $i=1,2,3$, then $M_3$ vanishes. So we only consider the case $|k_1| \sim |k_2| \gtrsim N$. 
The algebraic relation shows
\begin{align*}
|a_3+ \beta \lambda^{-2} b_3| \sim |k_1|^4 |k_3|
\end{align*}
Following these, we use the H\"{o}lder inequality and the Sobolev inequality to have 
\begin{align*}
& (\text{L. H. S. of (\ref{FTD_C})})  \lesssim N^{2s} \int  |\langle \p_x \rangle^{-2-s} u(t_0) |^2 | I u(t_0) | dx \\
& \hspace{0.3cm} \lesssim N^{2s} \| \langle k \rangle^{-2-s} u(t_0) \|_{L^4}^2 \| I u(t_0) \|_{L^2} 
\lesssim N^{2s} \| u (t_0) \|_{L^2}^2 \| m \widehat{u} (t_0) \|_{L_{\xi}^2},   
\end{align*}
which is bounded by $N^{2s} \| u (t_0) \|_{L^2}^3$ from the definition of $m$.

Secondly, we prove (\ref{FTD_Q}) when $|k_1| \geq |k_2| \geq |k_3| \geq |k_4|$. 
From (\ref{M_4}) and Sobolev's inequality, the left hand side of (\ref{FTD_Q}) is bounded by 
\begin{align*}
\Bigl| \Lambda_4 \Bigl( \frac{1}{ \prod_{i=1}^4 (N+|k_i|)^{2} m(k_i)  } \Bigr) (t_0) \Bigr|
& \lesssim N^{4s} \int | \langle \p_x \rangle^{-s-2} u (t_0) |^4 dx \\
& \lesssim  N^{4s} \| \langle \p_x \rangle^{-s-2} u(t_0) \|_{L^4}^4 \lesssim N^{4s} \| u (t_0) \|_{L^2}^4. 
\end{align*}
\end{proof}

Propositions~\ref{prop_ACL} and \ref{prop_FTD} imply that we can find a constant $C_3>0$ such that 
\begin{align} \label{es_ACL3}
\sup_{-N^{-5s} \leq t \leq N^{-5s}} \| I u (t) \|_{L^2} \leq C_3 \| I u(0) \|_{L^2}
\end{align}
For the details of the proof, see \cite{CoKeSt}.  
A direct calculation shows that 
\begin{align} \label{es_ACL4}
\| I u_{\lambda} (0, \cdot) \|_{L^2} \leq C_0 \lambda^{-s-7/2} N^{-s} \| u_0 \|_{ \dot{H}^s}
\end{align}
for some constant $C_0>0$. 
Here we take $\lambda \geq 1$ satisfying the following condition. 
\begin{align*}
\lambda^{-s-7/2} N^{-s} =\varepsilon_0 \ll 1.
\end{align*}
Then we combine (\ref{es_ACL3}) and (\ref{es_ACL4}) to have 
\begin{align*}
\sup_{-T \leq t \leq T} \| u(t) \|_{\dot{H}^s} \leq & \lambda^{7/2} \sup_{-\lambda^5 T \leq t \leq \lambda^5 T}
\| I u_{\lambda} (t) \|_{L^2} \\
\leq & C_3 \lambda^{7/2} \| I u_{\lambda} (0) \|_{L^2} 
\leq \varepsilon_0 C_1 C_3 \lambda^{-s} N^{-s} \| u_0 \|_{\dot{H}^s},
\end{align*}
when $\lambda^{5} T \leq N^{-5s}$. 
Therefore we have the following upper bound of the growth order of $\dot{H}^s$, 
\begin{align*}
\sup_{-T \leq t \leq T} \| u(t) \|_{\dot{H}^s} \leq C T^{7 /5 (2 s+5)} \| u_0 \|_{\dot{H}^s},
\end{align*}
for $-1 \leq s <0$.

\section{Proof of the ill-posedness}

In this section, we give the proof Theorem~\ref{ill_TK} which is based on \cite{Bo97}.  
From the argument to \cite{Ho}, it suffices to show that 
we seek for the initial data such that, for $|t|$ bounded, 
\begin{align} \label{es_ill}
\| A_3 (u_0) (t) \|_{\dot{H}^s} \lesssim \| u_0 \|_{ \dot{H}^s}^3,
\end{align}  
fails when $s<-3/2$. Here $A_3(u_0)$ is the cubic term of the Taylor expansion of the flow map as follows;
\begin{align} \label{def_cub}
A_3(u_0)(t) =2 \int_0^{t} U(t-s) \p_x(u_1(s) A_2 (u_0) (s)) ds,
\end{align}
where $u_1(t)=U(t) u_0$ and 
\begin{align*}
A_2(u_0) (t)= \int_0^t U(t-s) \p_x(u_1(s)^2 ) ds,
\end{align*}
which is the quadratic term of the Taylor expansion of the flow map. 
We put a sequence of initial data $\{ \phi_N \}_{N=1}^{\infty} \in H^{\infty}$ as follows;
\begin{align} \label{def_ini}
\widehat{\phi}_N(k)= N^{-s} (\chi_N (k) + \chi_{-N} (k) ).
\end{align}
Clearly $\| \phi \|_{\dot{H}^s} \sim 1$. 
A simple computation shows that 
\begin{align*}
\widehat{A}_2 (u_0) (t)= \sum_{k_1 \neq 0 ,k \neq k_1} 
k~\frac{ e^{i p(k) t} -e^{ i p(k_1)t+i p(k-k_1) t} }{ q_0(k_1,k-k_1) } \widehat{u}_0 (k_1) \widehat{u}_0(k-k_1),
\end{align*}
where 
\begin{align*}
q_0(k_1,k-k_1) := \frac{5}{2} k k_1 (k-k_1) \bigl\{ k^2+ k_1^2+(k-k_1)^2+\frac{6}{5} \beta  \bigr\}. 
\end{align*}
Substituting this into (\ref{def_cub}), we use the Fourier inversion formula to have
\begin{align} \label{cub_2}
A_3(u_0)(t)=& 2 \sum_{k_1 \neq 0} \sum_{k_2 \neq 0} \sum_{k_3 \neq 0} 
e^{i (k_1+k_2+k_3) x+ i p(k_1+k_2+k_3) t} \Bigl( -\frac{1-e^{-i q_1 t}}{q_1} +\frac{1-e^{-i q_2 t} }{q_2}
\Bigr)  \nonumber \\
& \hspace{0.5em} \times \frac{ (k_1+k_2+k_3) (k_2+ k_3) }{q_0(k_2,k_3) } 
\widehat{u}_0 (k_1) \widehat{u}_0 (k_2 ) \widehat{u}_0 (k_3),
\end{align}
where 
\begin{align*}
q_1:= & \frac{5}{2} (k_1+k_2) (k_1+k_3) (k_2+k_3) \bigl\{ (k_1+k_2)^2+ (k_1+k_3)^2+(k_2+k_3)^2+\frac{6}{5} \beta  \bigr\},\\
q_2:= & \frac{5}{2} k_1 (k_2+k_3) (k_1+k_2+k_3) \bigl\{ k_1^2+ (k_2+k_3)^2+(k_1+k_2+k_3)^2+\frac{6}{5} \beta  \bigr\}.
\end{align*}
Note that $q_2$ does not vanish but $q_1 $ vanishes when $k_1=-N$ and $k_2=k_3=N$.   
In this case, inserting (\ref{def_ini}) into (\ref{cub_2}), we obtain 
\begin{align*}
|\widehat{A}_3 (\phi_N) (t)| \gtrsim C_1 |t| N^{-3s-4} |k| \chi_{N}(k) -C_2 N^{-3s-8} \chi_{N}(k) 
\end{align*}
for some constants $C_1>0$ and $C_2 \geq 0$. So there exists $C_3>0$ such that 
\begin{align*}
\| A_3 (\phi_N) (t) \|_{\dot{H}^s} \gtrsim C_3 N^{-2s-3}
\end{align*} 
for $|t|$ bounded. From $\| \phi_N \|_{\dot{H}^s} \sim 1$, (\ref{es_ill}) fails for $s <-3/2$.

\end{document}